\documentclass[a4paper,11pt,oneside,reqno]{amsart} %<<<
\usepackage[width=15cm,height=22cm]{geometry}
\usepackage{amsfonts,amssymb,amsmath,amsthm}
\usepackage{graphicx}
\usepackage{hyperref}
\usepackage{pdfsync}
\usepackage{bbm}
\newtheorem{theorem}{Theorem}[section]
\newtheorem{lemma}[theorem]{Lemma}

\newtheorem{corollary}[theorem]{Corollary}

\theoremstyle{definition}
\newtheorem{definition}[theorem]{Definition}

\theoremstyle{remark}
\newtheorem{remark}{Remark}
\newtheorem{example}{Example}

\numberwithin{equation}{section}
\def\d{\mathrm{d}}

\def\<{\langle}
\def\>{\rangle}

\newcommand{\R}{\ensuremath{\mathbb{R}}}
\newcommand{\C}{\ensuremath{\mathbb{C}}}
\newcommand{\E}[1]{\mathbb{E}\left[#1\right]}
\newcommand{\Z}{\ensuremath{\mathbb{Z}}}

\newcommand{\cS}{\ensuremath{\mathcal S}}

\newcommand{\cU}{\ensuremath{\mathcal U}}

\newcommand{\bbE}{{\ensuremath{\mathbb E}} }

\newcommand{\bbP}{{\ensuremath{\mathbb P}} }

\def\1{\ifmmode {1\hskip -3pt \rm{I}}
\else {\hbox {$1\hskip -3pt \rm{I}$}}\fi} %indicator

\begin{document}

\title{Smallest Singular Value for Perturbations of Random Permutation Matrices}
\author[G. Ben Arous]{G\'{e}rard Ben Arous}
\address{G. Ben Arous\\
 Courant Institute of the Mathematical Sciences\\
 New York University\\
 251 Mercer Street\\
 New York, NY 10012, USA}
 \email{benarous@cims.nyu.edu}
\author[K. Dang]{Kim Dang}
\address{K. Dang\\
Department of Mathematics\\
Yale University\\
10 Hillhouse Ave\\
New Haven CT 06511}
\email{kim.dang@yale.edu}
\subjclass[2010]{15B52, 60B20, 60C05} \keywords{Random Matrices, Least Singular Values, Random Permutations, Random Walks with negative drift}

\date{\today}

\maketitle
\begin{abstract}We take a first small step to extend the validity of Rudelson-Vershynin type estimates to some sparse random matrices, here random permutation matrices. We give lower (and upper) bounds on the smallest singular value of a large random matrix $D+M$ where $M$ is a random permutation matrix, sampled uniformly, and $D$ is diagonal. When $D$ is itself random with i.i.d terms on the diagonal, we obtain a Rudelson-Vershynin type estimate, using the classical theory of random walks with negative drift.
\end{abstract}

\section{Introduction}

If $M$ is large random matrix, it is both important and usually difficult to find sharp lower bounds on its smallest singular value  $s_{\min}(M)$ (see  \cite{rudelson}, \cite{rudelson06}, \cite{rudelsonvershynin08}, \cite{rv10}, \cite{tao_vu}, \cite{tv09}, \cite{tv10-2}). For instance, such lower bounds were important for the proofs of the circular law (see \cite{tv10}, \cite{gotze}, \cite{bordenave},  \cite{rudelson}, \cite{rudelsonvershynin08}, \cite{tv09}, \cite{tao_vu}, \cite{tv10}), or the single ring theorem \cite{guionnet}.

In \cite{rudelson_vershynin}, Rudelson and Vershynin give remarkable quantitative estimates of the smallest singular value for perturbation of random unitary or orthogonal matrices, i.e. for matrices $M+D$ where $M$ is a random unitary (or orthogonal) matrix, and $D$ is a fixed matrix.  In this work we explore a possible extension of these estimates to the same question in the case where $M$ is sampled from a discrete subgroup. The tools  in \cite{rudelson_vershynin} relies on the Lie structure of the unitary and orthogonal groups. These tools are not readily available for discrete subgroups of these groups.

In this paper, we will consider a simple example of the case where $M$ is sampled uniformly from a discrete subgroup group of the unitary group, i.e. the case where $M$ is a random permutation matrix, sampled uniformly, and $D$ is diagonal.

We first prove sharp deterministic estimates for the smallest singular value $s_{\min}(M+D)$, where $M$ is a permutation matrix and $D$ is diagonal. The interesting situation is the case where $D$ has (diagonal) entries both inside and outside the unit circle. Indeed it is easy to see that, if the entries of $D$ all lie outside (or all lie inside) the disk of radius $1$, then the smallest singular value $s_{\min}(M+D)$ can be bounded below by the smallest distance of these entries to the unit circle.  
We use those deterministic estimates to show in particular that, if the diagonal entries of $D$ are themselves random (and i.i.d), and $M$ is a random permutation matrix, then a Rudelson-Vershynin type estimate holds, under natural assumptions on the law of the entries of $D$. Our proof uses a new result in the classical theory of random walks with negative drift, given in the appendix
\ref{appendix}.\\

Acknowledgements: The authors are grateful to Van Vu for sharing this interesting question, and to Jean Bertoin and Amir Dembo for extensive discussions about the theory of random walks with negative drifts. The work of the first author has been partially supported by NSF Grant, DMS-1209165.

\subsection{Statements of results}\label{statements}

Let $\sigma\in\cS_N$ and define $M_\sigma$ to be the $N\times N$ permutation matrix with entries 

\begin{equation}
M_\sigma(i,j)=\mathbbm{1}_{\sigma(i)=j}, \qquad \text{for } 1\leq i,j\leq N.
\end{equation}

For any $N$-tuple of complex numbers $d_1,\dots, d_N$, consider the diagonal matrix 
\begin{equation}
D=\text{diag}(d_1,\dots, d_N).
\end{equation}

We would like to understand the invertibility and the behavior of the minimum singular value $s_{\min}(A)$ of the matrix
\begin{equation}
A=D+M_\sigma.
\end{equation}

Using the cycle decomposition of the permutation $\sigma$, the matrix $A$ is easily reduced to a block-diagonal matrix by a unitary conjugation.
The study of the smallest singular value of $A$ then amounts to studying the smallest singular values of the matrix blocks, given by each cycle of $\sigma$.

Indeed, consider the cycle decomposition of the permutation 
\begin{equation}\sigma=(C_1, \dots, C_{K(\sigma)}),\end{equation} 

where $K(\sigma)$ denotes the total number of cycles. Define 
\begin{equation}N_1=|C_1|, \dots, N_{K(\sigma)}=|C_{K(\sigma)}|\end{equation} 
to be the cycle lengths. We will assume, without loss of generality, that the cycles have been ranked by decreasing length, i.e. 
\begin{equation}
N_1 \geq N_2 \geq  \dots \geq N_{K(\sigma)}
\end{equation}

For $1\leq i\leq K(\sigma)$, write the cycle $C_i$ as 
\begin{equation}
C_i=(n_i, \sigma(n_i),  \dots,  \sigma ^{N_i-1}(n_i)),
\end{equation}
where $1\leq n_i\leq N$ is the number starting the cycle $C_i$.

Now, for $1 \leq i \leq K(\sigma)$, denote by $A_i$ the $N_i\times N_i$ matrix defined by
\begin{equation}
A_i := D_i + U_{N_i},
\end{equation}

where  $D_i$ is the diagonal matrix 
\begin{equation}
D_i=\text{diag}(d_{n_i}, d_{\sigma(n_i)}, \dots, d_{\sigma^{N_i-1}(n_i)})
\end{equation}

and where for any integer $n \geq 1$,  the $n\times n$ matrix $U_n$ is defined by 
\begin{equation}\label{U_n}
U_n=\begin{pmatrix}0 & 1 & \dots & 0 \\ \vdots & \ddots & \ddots & 0 \\ \vdots &   & \ddots & 1 \\ 1 & \dots & \dots & 0\end{pmatrix}
\end{equation}

We have the following simple result, which settles the invertibility question and reduces the estimation of $s_{\min}(A)$ to the same question for the matrices $A_i$  for $1\leq i \leq K(\sigma)$, which pertain to the case of single-cycle permutations.

\begin{theorem}\label{CycleDecomposition} \
\begin{enumerate} 

\item $A$ is invertible iff, for every $1\leq i\leq K(\sigma)$,
\begin{equation}
\prod_{\ell\in C_i}d_\ell \neq (-1)^{N_i}.
\end{equation}

\item The smallest singular value of $A$ is given by
\begin{equation}
s_{\min}(A)=\min_{1\leq i \leq K(\sigma)}s_{\min}(A_i).
\end{equation}
\end{enumerate}
\end{theorem}

Thus, the invertibility of $A$ reduces to understanding the behavior of the products $\prod_{\ell\in C_i}d_\ell$ of the diagonal elements of $D$ on the cycles of the permutation $\sigma$. It is clear that $A$ is invertible in the case where the modulus of those diagonal elements are either all smaller than $1$ or all larger than $1$.\\

We start by a theorem showing that for this case, the matrix $A$ is indeed well invertible, i.e. that the least singular value of $A$ is bounded away from zero.

\begin{theorem}\label{OneSided} \
\begin{enumerate}
\item
Assume that, for all $1\leq i\leq N$, $|d_i| < 1$. Define $\epsilon_N = 1 - \max_{1 \leq i \leq N} |d_i|$.
Then, we have the following lower bound 
\begin{equation}
s_{\min}(A) \geq  \frac{1}{2\sqrt2} \epsilon_N
\end{equation}

\item
Assume that, for all $1\leq i\leq N$, $|d_i| > 1$. Define $\epsilon_N = \min_{1 \leq i \leq N} |d_i| -1$.
Then, we have the following lower bound 
\begin{equation}
s_{\min}(A) \geq  \frac{1}{2\sqrt2} \epsilon_N
\end{equation}
\end{enumerate}
\end{theorem}

In particular, we see that
\begin{corollary} \label{OneSided2} \
If $|d_i| < 1- \epsilon$, for all $ 1 \leq i \leq N$, or if $|d_i| > 1+ \epsilon$, for all $1 \leq i \leq N$, then the least singular value of $A$ is bounded away from zero, independently of $N$:
\begin{equation}
s_{\min}(A) \geq  \frac{1}{2\sqrt2} \epsilon
\end{equation}
\end{corollary}

It might be useful to give here the simplest possible and most explicit example, i.e. the case where the matrix $D$ is scalar. An explicit computation of $s_{min}(A)$ is then easy.

\begin{example}\label{Example-Scalar-Case}
Let $d \in \C$. If $D=dI_N$,  the smallest singular value of $A$ is explicitly given by the formula
\begin{equation}
s_{\min}(A)=\inf_{1\leq i\leq K(\sigma)}\varphi_{N_i}(d),
\end{equation}

where, for any integer $n \geq 1$ and $z \in \C$,
\begin{equation}\label{varphi}
\varphi_n(z):=\text{dist}(-z, \cU_n) = \min ( |\omega + z|, \omega \in \cU_n)
\end{equation}
and $\cU_n$ denotes the set of $n$-th roots of unity
\begin{equation}
\cU_n = \{  \omega \in \C, \omega^n=1 \}.
\end{equation}

 It is very easy to see that the following elementary estimate holds:
\begin{equation}
| |z| -1| \leq \varphi_n(z) \leq | |z| -1| + 2( |z| \wedge 1) \sin(\frac{\pi}{2n})
\end{equation}
Thus, we see here that  one should indeed distinguish between the cases where $|d|=1$ and $|d|\neq 1$. If $|d|\neq1$,
\begin{equation}
||d|-1| \leq s_{\min}(A) \leq ||d|-1|+ 2 (|d|\wedge 1) \sin(\frac{\pi}{2N_1}).
\end{equation}

If $N \to \infty$, and if $d$ depends on $N$, i.e. $d=d_N$, we can see that the asymptotic behavior of $s_{min}(A)$ is simple.
$s_{min}(A)$ is bounded below by $\epsilon_N= ||d_N|-1|$, and this lower bound is sharp when the largest cycle length diverges.

But, if $|d|=1$ and $d=e^{i\eta}$, then the arithmetic properties of $\eta$ become important. We will not dwell on that here.
\end{example}

\bigskip

We will now study the more interesting case where the $|d_{\ell}|$'s can take values both above and below $1$, and give sharp deterministic bounds on $s^2_{min}(A)$. To state these, we introduce some notations.
\begin{definition} \label{definitionrho}

Let $D= \text{diag}(d_1,\dots, d_N)$. 
\begin{enumerate}
\item Let
\begin{equation}
c_0(D)= |\det(D+U_N)|^2= | (-1)^N - \Pi_{l=1}^N d_l|^2.
\end{equation}

\item
For $1 \leq k \leq m \leq N$, let
\begin{equation}
\beta_{k,m}(D) := (-1)^{m-k+1} \Pi_{l=k}^m d_l
\end{equation}
For $2 \leq k \leq N+1$ and $ m = k-1$, we set $ \beta_{k,m} = 1$.

\item
Let
\begin{equation}
\gamma_N(D) = \max_{1\leq k \leq N} \sum_{m=1}^N |\beta_{k,m}|^2 .
\end{equation}

\item
Assuming that $D+U_N$ is invertible, i.e. that $c_0(D) > 0$, we define 
\begin{equation}
\rho_N(D) =2 ( \rho_N^{(1)}(D) + \rho_N^{(2)}(D))
\end{equation}
where 
\begin{equation}
\rho^{(1)}_N(D) = \frac{1}{c_0(D)} \sum_{k=1}^N \ |\beta_{1, k-1}|^2 \sum_{m=1}^N \ |\beta_{m+1,N}|^2
\end{equation}
and 
\begin{equation}
\rho^{(2)}_N(D) = \sum_{1\leq k\leq m-1 \leq N} |\beta_{k+1,m-1}|^2
\end{equation}
\end{enumerate}
\end{definition}

We have then the following bounds

\begin{theorem}\label{ExplicitLowerBounds-1Cycle} \
\begin{enumerate}

\item Assuming that $D+U_N$ is invertible, i.e. that $c_0(D) > 0$, 
\begin{equation}
\frac{1}{\rho_N(D)} \leq s^2_{min}(D+U_N) \leq \frac{c_0(D)}{\gamma_N(D)}
\end{equation}

\item Assuming that $A$ is invertible, i.e. that $c_0(D_i) > 0$, for $1\leq i\leq K(\sigma)$, then 
\begin{equation}
\min_{1\leq i \leq K(\sigma)} \frac{1}{\rho_N(D_i)} \leq s^2_{min}(A) \leq  \min_{1\leq i\leq K(\sigma)}\frac{c_0(D_i)}{\gamma_N(D_i)}
\end{equation}

\end{enumerate}
\end{theorem}

\bigskip
We then use the estimates given above to study first the case where the permutation $\sigma$ is fixed and the diagonal matrix $D$ is chosen randomly, with i.i.d entries. We will then very easily translate our results to the case of a random permutation matrix.

We consider the diagonal elements $d_{\ell}$ to be i.i.d random variables sampled from a common probability distribution $\mu$ on the complex plane $\C$. We have given a lower bound on  $s^2_{min}(A)$ in Theorem~\ref{OneSided} when $\mu$ has support in $|z| <1$ or in $|z| >1$. Here, we treat the more interesting case where $\mu$ gives mass to both $|z|<1$ and $|z|>1$.
 We need to distinguish between the cases where $m \neq 0$ and $m=0$, where
\begin{equation}
m := \int \log|x|d\mu(x)=\E{\log|d_{\ell}|}.
\end{equation}
In the case where $m \neq 0$, we can (and will) assume, without loss of generality, that $m < 0$.
Indeed, using the simple Lemma~\ref{duality} given below, the case where $m>0$ is entirely analogous to the case where $m<0$ and in fact can be derived as a simple consequence.

We will need the following assumptions, on the measure $\mu$.
\begin{enumerate}\label{hypothesis1-4}
\item[H1.] We will assume that the support of $\mu$ intersects both 
$\{z \in \C, |z| < 1\}$ and $\{z \in \C, |z| < 1\}$. We also assume, for simplicity, that it is bounded away from zero .
\item[H2.]
We assume that $\mu$ has finite moments, i.e. there exists a $t>1$ such that, 
$\int |z|^t \d\mu(z) = \bbE[|d_{\ell}|^t] < \infty$.
\item[H3.]
$m<0$.
\item[H4.] There exist $C>0$ and $\rho \in (0,1]$ such that, for all $h>0$,
\begin{equation}
\sup_{x >0} \mu(\{z \in \C | x <  |z| < x+h\}) \leq C h^{\rho}.
\end{equation} 
\end{enumerate}

Our main result is that, under these assumptions, the least singular value $s_{min}(A)$ decays to zero as a negative power of $N$, when $N$ is large, up to logarithmic corrections (we believe these corrections are merely technical conveniences and should not be relevant). We also show that the order of magnitude of $s^2_{min}(A)$ depends on the value of the unique positive number $\theta$ such that
\begin{equation}
\int |x|^{2\theta} d\mu(x)= 1.
\end{equation}

We will also use the following notation. For a permutation $\sigma$, and $k>0$, define 
\begin{equation}
L(k, \sigma) = \sum_{i=1}^{K(\sigma)} e^{- k N_i}
\end{equation}

\begin{theorem}\label{NegativeDrift1}
Under the assumptions H1-H4 above, we consider first the case where $\theta<1$. 
There exist constants $k_0 >0$ and $\delta_0 >0$, such that for every $\delta<\delta_0$ and for every $t>0$,
\begin{enumerate}
\item 
\begin{equation}
\bbP_D [ s^2_{min}(A) \leq t] \leq C(\delta) L(k_0,\sigma) t^{\delta} + CN t^{\theta} (\log\frac{1}{t})^{\theta}
\end{equation}

\item
This obviously implies the Rudelson-Vershynin type estimate
\begin{equation}
\bbP_D [ s^2_{min}(A) \leq t] \leq C N t^{\alpha}
\end{equation}
with any $\alpha < \min(\theta, \delta)$

\item If we consider a sequence of permutations $\sigma_N$ such that $L(k,\sigma_N) = o( N^{\frac{\delta}{\theta}})$, then the sequence of distributions of the random variables $ (N^{\frac{1}{\theta}} \log N s^2_{min}(A))^{-1} $ is tight.
\begin{equation}
\limsup_{ N \to \infty} \bbP_D [ s^2_{min}(A) \leq \frac{u}{N^{\frac{1}{\theta}} \log N}] \leq C u^{\theta}(\log \frac{1}{u})^{\theta}
\end{equation}

\item In particular, if the permutations $\sigma_N$ are chosen randomly, say under the uniform measure $\bbP_N$ on the symmetric group, then the sequence of distributions of the random variables $ (N^{\frac{1}{\theta}} s^2_{min}(A))^{-1} $ is tight under the product measure $\bbP_N \times \bbP_D$.
\begin{equation}
\limsup_{ N \to \infty} \bbP_N \times \bbP_D [ s^2_{min}(A) \leq \frac{u}{N^{\frac{1}{\theta}} \log N}] \leq C u^{\theta}(\log \frac{1}{u})^{\theta}
\end{equation}

\end{enumerate}
\end{theorem}

\begin{remark}
For a fixed distribution $\mu$, the bound we obtain in the first item of this result, depends only on the conjugation class of $\sigma$, i.e. its cycle structure as it obviously should, since the distribution of the diagonal entries is exchangeable.
The constants $k_0>0$ and $\delta_0$ will be described below. They depend on the distribution $\mu$.
Moreover, using the fact that $\sum_{i=1}^{K(\sigma)} N_i =N$ and the inequality between an arithmetic and a geometric mean, we see easily that
\begin{equation}
K(\sigma) e^{-k\frac{N}{K(\sigma)}} \leq L(k,\sigma) \leq K(\sigma) \leq N
\end{equation}
\end{remark}

We now study the case where $\theta >1$.

\begin{theorem}\label{NegativeDrift2}
Under the assumptions H1-H4 above, and if $\theta > 1$,
\begin{enumerate}
\item
 There exist constants $k_0 >0$ and $\delta_0 >0$, such that for every $\delta<\delta_0$, $T>0$, and for every $t < \frac{T}{N \log N}$, 
\begin{equation}
\bbP_D [ s^2_{min}(A) \leq t] \leq C(\delta) L(k_0,\sigma) t^{\delta} + C(T) N t^{\theta} (\log \frac{1}{t})^{\theta}
\end{equation}
\item
This implies the Rudelson-Vershynin type estimate
\begin{equation}
\bbP_D [ s^2_{min}(A) \leq t] \leq C N t^{\alpha}
\end{equation}
with any $\alpha < \min(\delta, 1)$

\item If we consider a sequence of permutations $\sigma_N$ such that $L(k,\sigma_N) = o( N^{\frac{\delta}{\theta}})$, then the sequence of random variables $ (N \log N s^2_{min}(A))^{-1} $ converges to zero in probability (and even a.s if $\theta >2$). Indeed
\begin{equation}
\bbP_D [ s^2_{min}(A) \leq \frac{u}{N \log N}] \leq C \frac{u^{\theta}}{N^{\theta-1}}
\end{equation}
\item In particular, if the permutations $\sigma_N$ are chosen randomly, say under the uniform measure $\bbP_N$ on the symmetric group, then the sequence of distributions of the random variables $ (N \log N s^2_{min}(A))^{-1} $ is tight under the product measure $\bbP_N \times \bbP_D$.
\begin{equation}
\limsup_{ N \to \infty} \bbP_N \times \bbP_D [ s^2_{min}(A) \leq \frac{u}{N \log N}] \leq C u^{\theta}(\log \frac{1}{u})^{\theta}
\end{equation}

\end{enumerate}
\end{theorem}

We also give upper bounds on $s^2_{min}(A)$, which show no transition with the value of $\theta$, but are probably sharp only for $\theta <1$.

\begin{theorem}\label{NegativeDrift3}
Under the assumptions H1-H3 above, and for any value of the exponent $\theta>0$, $s^2_{min}(A)$ is at most of order $N^{- \frac{1}{\theta}}$.
More precisely,
\begin{enumerate}
\item We have the following estimate for the upper-tail of $s_{min}(A)$. There exist two constants $k>0$ and $C>0$, and, for any $T>0$ there exists a constant $C(T)$, such that for $0< t < T$,
\begin{equation}
\bbP_D[ s^2_{min}(A) \geq t] \leq C(T) e^{-NCt^{\theta}}
\end{equation}
\item
In particular, for any $u>0$,
\begin{equation}
\bbP_D[ s^2_{min}(A) \geq \frac{u}{N^{\frac{1}{\theta}}}] \leq  e^{-Cu^{\theta}}
\end{equation}
\end{enumerate}
\end{theorem}

\begin{remark}
A closer analogue to the Rudelson-Vershynin result would be to study the case where the permutation $\sigma$ is chosen randomly, and the diagonal matrix $D$ is fixed, say under the assumption that the empirical distribution $\mu_N= \frac{1}{N} \sum_{i=1}^N \delta_{d_i}$ is close to a measure $\mu$, with appropriate assumptions on the measure $\mu$. If the permutation $\sigma$ is chosen uniformly at random, for cycles $C_i$ of length $N_i$ much smaller than $N$, the case of deterministic $D_i$ should be close to the case where the diagonal entries are sampled in an i.i.d fashion from the empirical distribution $\mu_N= \frac{1}{N} \sum_{i=1}^N \delta_{d_i}$. Indeed sampling $N_i$ elements from the measure $\mu_N$ with or without repetition should not make much difference, if $N_i$ is much smaller than $N$. For cycles of length comparable to $N$, this argument does not hold obviously, and the case of random $D$ (sampling without repetition) is then simpler than the case of fixed deterministic $D$.
\end{remark}

\subsection{Organization of the paper}
This paper is organized as follows: We give the proof of Theorem~\ref{CycleDecomposition} and the explicit computation of Example~\ref{Example-Scalar-Case} in Section~\ref{sectionCD}. The proof of Theorem~\ref{ExplicitLowerBounds-1Cycle} is given in Section~\ref{sectionDB} in two steps. The lower bounds are given in Section~\ref{LowerBounds}, the upper bounds in Section~\ref{UpperBounds}. Based on Theorem~\ref{ExplicitLowerBounds-1Cycle}, the proof of Theorem~\ref{OneSided} is given in Section~\ref{section_proof_onesided}. In Section~\ref{sectioniid}, we give the proofs of the results where the entries of $D$ are i.i.d., the proofs of Theorem~\ref{NegativeDrift1} and Theorem~\ref{NegativeDrift2} are given in Section~\ref{negd1},  the proof of Theorem~\ref{NegativeDrift3} in Section~\ref{negd3}.

\section{Cycle decomposition}\label{sectionCD}
\subsection{Proof of Theorem~\ref{CycleDecomposition}}
In this section, we give the proof of Theorem~\ref{CycleDecomposition}. We start by providing the following block-decomposition result, which obviously implies Theorem~\ref{CycleDecomposition}.
\begin{theorem}\label{CycleDecomposition2}\
\begin{enumerate}
\item [a)]A is unitarily conjugate to the block-diagonal matrix $ \text{diag}(A_1, \dots, A_{N_{K(\sigma)}})$.
More precisely, there exists a permutation $\tau$ such that 
\begin{equation}
M_{\tau}AM_{\tau}^{-1}= diag(A_1, \dots, A_{N_{K(\sigma)}})
\end{equation}
\item [b)]The determinant of $A$ is given by
\begin{equation}
\det(A)=\prod_{i=1}^{K(\sigma)}\left(\prod_{\ell\in C_i}d_\ell - (-1)^{N_i}\right)
\end{equation}
\end{enumerate}
\end{theorem}

\begin{proof}[Proof of Theorem~\ref{CycleDecomposition2}]

Consider the permutation $\sigma^{ordered}$ given by its cycle decomposition

\begin{equation}
 \sigma^{ordered} = ( (1, \dots, N_1), (N_1+1, \dots, N_1+N_2), \dots, (N_1+\dots+N_{K(\sigma)-1}+1, \dots, N) )
 \end{equation}
 
 $\sigma$ and $\sigma^{ordered}$ have the same cycle lengths, so they are in the same conjugation class.
 Thus, there exists a permutation $\tau$ such that
 \begin{equation}
 \sigma = \tau \sigma^{ordered} \tau^{-1}
 \end{equation}
 
 It is in fact very simple to write explicitly the permutation $\tau$:
 \begin{equation}
 \tau(1)= n_1, \tau(2)= \sigma(n_1), \dots, \tau(N_1)= \sigma^{N_1-1}(n_1)
 \end{equation}
 Similarly, for any $1\leq i \leq K_{\sigma}$ and any $ 1 \leq k \leq N_i$
 \begin{equation}
 \tau(N_1+\dots+N_i+k)= \sigma^{k-1} (n_i)
 \end{equation}

Thus, the permutation matrix $M_{\sigma}$ can be written as
\begin{equation}
M_{\sigma} = M_{\tau}^{-1}M_{\sigma^{ordered}}M_{\tau}
\end{equation}

and we have

\begin{equation}
M_{\tau} A M_{\tau}^{-1}= M_{\tau} D M_{\tau}^{-1} + M_{\sigma^{ordered}}
\end{equation}

But it is obvious that
\begin{equation}
M_{\tau} D M_{\tau}^{-1} = \text{diag} (D_1, \dots, D_{K(\sigma)})
\end{equation}

and

\begin{equation}
M_{\sigma^{ordered}} =  \text{diag} (U_{N_1}, \dots, U_{N_{K(\sigma)}}),
\end{equation}

so that
\begin{equation}
M_{\tau} A M_{\tau}^{-1}= \text{diag} (A_1, \dots, A_{K(\sigma)})
\end{equation}
This proves the first item a) of Theorem~\ref{CycleDecomposition2}.

The second item b) is then obvious, if one notes the simple fact that
\begin{equation}
\det ( \text{diag}(d_1,\dots, d_n) +U_n) = \prod_{k=1}^n d_k - (-1)^n
\end{equation}
\end{proof}

\subsection{The scalar case}

In this section, we give the explicit computation presented in  Example~\ref{Example-Scalar-Case}:

\begin{lemma}\label{lemma-scalar-case}
 \begin{equation}
s_{\min}(D+U_N)=\varphi_N(d).
\end{equation}
\end{lemma}

\begin{proof}
Note that the spectrum of $U_N$ is very explicit. It consists of the set of N-th roots of unity, $\cU_N$.
Since $U_N$ and $U_N^*$ commute, the spectrum of $\overline{d}U_N + d U^{*}_N $ is also easy to compute:
\begin{equation}
\text{Spectrum}(\overline{d}U_N + dU^{*}_N) = \{ \overline{d} \omega + d \overline{\omega}, \omega \in \cU_N \} = \{ 2 \Re(\overline{d}\omega), \omega \in \cU_N \}.
\end{equation}

But
\begin{equation}
(D+U_N)(D+U_N)^{*} = (1+|d|^2) Id_N + \overline{d}U_N + dU^{*}_N
\end{equation}

Thus, the set of singular values of $(D+U_N)$, or equivalently the spectrum of $(D+U_N)(D+U_N)^*$ is given by

\begin{eqnarray}
\text{singvalues}(D+U_{N})&=&\{1+|d|^2+2\text{Re}(\overline{d}\omega), \omega\in\cU_{N}\}\\
&=&\{|d+\omega|^2, \omega\in\cU_N\}.
\end{eqnarray}

Then,
\begin{equation}
s_{\min}^2(D+U_N) = \inf_{\omega\in\cU_{N}}|d+\omega|^2= \text{dist}^2(-d, \cU_{N}),
\end{equation}

so that by definition of $\varphi_N(d)$ in \eqref{varphi},
\begin{equation}
s_{\min}(D+U_N)=\varphi_N(d).
\end{equation}
\end{proof}

Using the last point of Theorem~\ref{CycleDecomposition} and Lemma~\ref{lemma-scalar-case}, it is now easy to complete the computation given in Example~\ref{Example-Scalar-Case}.

\section{Deterministic Bounds}\label{sectionDB}
\subsection{Explicit inversion of the Matrix $A$}

In this section, we give an explicit inversion of the matrix $A$. By Theorem~\ref{CycleDecomposition}, the problem reduces to an explicit inversion of the matrix $D+U_N$. We first introduce the following notations. 

\begin{definition}
Assume that the matrix $D+U_N$ is invertible, i.e. that $(-1)^N \prod_{k=1}^N d_k \neq 1$.
\begin{enumerate}
\item Denote 
\begin{equation}
\sigma_0(k):=(k+1),
\end{equation}
for a single-cycle permutation $\sigma_0=(1 \dots N)$, i.e. $\sigma_0(k)=k+1$ for $1\leq k\leq N-1$ and $\sigma_0(N)=1$.
\item Denote by $B(D)$ the $N \times N$ strictly lower triangular matrix defined by its entries as follows.
For $ 1 \leq j \leq i \leq N$
\begin{equation}
B(D)_{i,j}= \beta_{j+1, i-1} \1_{j \leq i-1}
\end{equation}
\item Define the rank-one matrix 
\begin{equation}
C(D) = \frac{1}{1 - (-1)^N \Pi_{l=1}^N d_l} E(D) F(D)^T
\end{equation}
where $E(D)$ and $F(D)$ are the following  column N-vectors.
For $1 \leq i,j \leq N$,
\begin{equation}
 E(D)_i= \beta_{1,i-1} \qquad \text{and} \qquad F(D)_j = \beta_{j+1,N}
 \end{equation} 
 
\end{enumerate}
\end{definition}

\begin{theorem}\label{Inversion} \

\begin{itemize}
\item [a)]
When the matrix $D+U_N$ is invertible, its inverse is given by 
\begin{equation}
(D+U_N)^{-1}= B(D)+C(D)
\end{equation}

\item [b)] When the matrix $A$ is invertible, its inverse is given by 
\begin{equation}
A^{-1} =  M_{\tau}^{-1} \text{diag}(B(D_1)+C(D_1), \dots, B(D_{K(\sigma)})+C(D_{K(\sigma)}))M_{\tau}
\end{equation}
\end{itemize}
\end{theorem}

\begin{proof}[Proof of Theorem~\ref{Inversion}]

As a first step, we give an explicit inversion for the matrix $D+U_N$, assuming that it is invertible, i.e that $ \prod_{k=1}^N d_k \neq (-1)^N $.

Given a vector $y\in\C^N$, we want to find the vector $x\in\C^N$ such that 
\begin{equation}\label{explicit_inversion_1}
(D+U_N)x=y,
\end{equation}

i.e. such that, for $1\leq k\leq N$
\begin{equation}
d_kx_k+x_{(k+1)}=y_k
\end{equation}

We will denote $d'_k=-d_k$ for ease of notations. \eqref{explicit_inversion_1} is equivalent to the $N$ equations
\begin{eqnarray}\label{explicit_inversion2}
x_2&=&d'_1x_1+y_1\nonumber\\
x_3&=&d'_2d'_1x_1+d'_2y_1+y_2\nonumber\\
\vdots \nonumber\\
x_L&=&d'_{L-1}\dots d'_1x_1+d'_{L-1}d'_{L-2}\dots d'_2y_1+\dots +d'_{L-1}y_{L-2}+y_{L-1}\nonumber\\
\vdots\nonumber\\
x_N&=&d'_{N-1}\dots d'_1x_1+d'_{N-1}\dots d'_2y_1+\dots+d'_{N-1}y_{N-2}+y_{N-1}\nonumber\\
x_1&=&d'_N\dots d'_1x_1+d'_N\dots d'_2y_1+\dots +d'_Ny_{N-1}+y_N
\end{eqnarray}

Recall that, for $1\leq k\leq m\leq N$,
\begin{equation}
\beta_{k,m}=\prod_{\ell=k}^md'_\ell
\end{equation}

and $\beta_{k,m}=1$ if $2\leq k\leq N+1$ and $m=k-1$. 
Then the last equation of \eqref{explicit_inversion2} can be solved for $x_1$:
\begin{equation}\label{explicit_inversion3}
x_1=\frac 1{1-\prod_{i=1}^Nd'_i}\sum_{k=1}^N\beta_{k+1,N}y_k
\end{equation}

From \eqref{explicit_inversion3}, we easily get the other components of the vector $x$. For $2\leq L\leq N$,
\begin{equation}\label{explicit_inversion4}
x_L=\beta_{1,L-1}x_1+\sum_{j=1}^{L-1}\beta_{j+1, L-1}y_j
\end{equation}

A simple inspection shows that the last two formulae \eqref{explicit_inversion3} and \eqref{explicit_inversion4} are equivalent to the fact that
\begin{equation}
x = (C(D)+ B(D)) y,
\end{equation}
 or equivalently, this proves the first point a) of Theorem~\ref{Inversion}, i.e. that
\begin{equation}
(D+U_N)^{-1} = C(D)+ B(D)
\end{equation}

Now, using the block-decomposition given in a) of Theorem~\ref{CycleDecomposition2}, one sees that the second point b) of Theorem~\ref{Inversion} is also proved.
\end{proof}

As an immediate corollary, we can compute the smallest singular value $s_{min}(D+U_N)$ by the largest singular value of the matrix $B(D)+C(D)$, or equivalently by its operator norm $ || B(D) + C(D)||$.
\begin{corollary}\label{OperatorNorm} \
\begin{itemize}
\item [a)] The least singular value of the matrix $D+U_N$ is given by
\begin{equation}
s_{min}(D+U_N) = \frac{1}{ || B(D) + C(D)||}
\end{equation}

\item [b)] The least singular value of the matrix $A$ is given by
\begin{equation}
s_{min}(A) =\min_{1\leq i\leq K(\sigma)}  \frac{1}{ || B(D_i) + C(D_i)||}
\end{equation}
\end{itemize}

\end{corollary}

\begin{proof}[Proof of Corollary~\ref{OperatorNorm}]

Obviously 
\begin{equation}
s_{min}^{-1}(D+U_N) = s_{max}(D+U_N)^{-1} = || (D+U_N)^{-1}|| = || C(D) + B(D) ||
\end{equation}
Which proves the first point of Corollary~\ref{OperatorNorm}.
The second point b) of Corollary~\ref{OperatorNorm} is then a direct consequence again of the second point of Theorem~\ref{CycleDecomposition}.

\end{proof}

This explicit expression is not easily computed in general. In fact computing an operator norm is usually as hard as computing a smallest singular value. But, the proof of Theorem~\ref{OneSided} will indeed rely on these estimates in terms of operator norms. Moreover, weakening these estimates using the Hilbert-Schmidt norms, we will prove, in the next section, the lower bounds given in Theorem~\ref{ExplicitLowerBounds-1Cycle}.

\subsection{ Lower Bounds on the least singular value}\label{LowerBounds}

In this section we prove deterministic lower bounds on $s_{min}(D+U_N)$ and $s_{min}(A)$. 

\begin{theorem}\label{LBrho}\
\begin{enumerate}

\item Assuming that $D+U_N$ is invertible, i.e. that $c_0(D) > 0$, 
\begin{equation}
 s^2_{min}(D+U_N) \geq \frac{1}{\rho_N(D)}
\end{equation}

\item Assuming that $A$ is invertible, i.e. that $c_0(D_i) > 0$, for $1\leq i\leq K(\sigma)$, 
\begin{equation}
 s^2_{min}(A) \geq \min_{1\leq i \leq K(\sigma)} \frac{1}{\rho_N(D_i)}
 \end{equation}

\end{enumerate}
\end{theorem}

\begin{proof}
Again we prove only the first part of the theorem, since the second follows immediately from the first part and from Theorem~\ref{CycleDecomposition}.
The result is a direct consequence of the Corollary~\ref{OperatorNorm}, and of the trivial bounds:
\begin{equation}
|| C(D)+B(D)||^2 \leq 2||C(D)||^2 + 2||B(D)||^2 \leq 2||C(D)||^2 + 2|| B(D)||_{HS}^2
\end{equation}
Indeed, since $C(D)$ is of rank one, its operator norm is equal to its Hilbert-Schmidt norm, and is given by
\begin{equation}
||C(D)||^2 = \rho^{(1)}_N(D)
\end{equation}
and the Hilbert-Schmidt norm of $B(D)$ is given by
\begin{equation}
|| B(D)||_{HS}^2= \rho^{(2)}_N.
\end{equation}
This shows that

\begin{equation}
s_{min}^2(D+U_N) = || C(D)+B(D)||^{-2} \geq \frac{1}{\rho_N(D)}
\end{equation}
which proves the first part of the theorem.
\end{proof}

\begin{remark}
Note that, obviously 
\begin{equation}
\gamma_N(D) \leq  \rho^{(2)}_N(D) \leq N \gamma_N(D).
\end{equation}

\end{remark}

\subsection{ Upper Bounds on the least singular value}\label{UpperBounds}

In this section we prove deterministic upper bounds on $s_{min}(D+U_N)$ and $s_{min}(A)$, given in Theorem~\ref{ExplicitLowerBounds-1Cycle}. Also, weaker upper bounds are given in Lemma~\ref{Lemma-Comparison-to-scalar} below, where the least singular values $s_{min}(D+U_N)$ and $s_{min}(A)$ are compared to the scalar case.
We will need the following notations.

\begin{definition}\label{delta}\
\begin{enumerate}
\item  Denote by $u(D)$ an $N$-th root of $ \Pi_{\ell=1}^{N} d_\ell$, i.e.
\begin{equation}
u(D)^N = \Pi_{\ell=1}^{N} d_\ell
\end{equation}

\item  Define $\delta_1 :=1$, and for $2 \leq k \leq N$, let
\begin{equation}
\delta_k(D) := \frac{\Pi_{\ell=1}^{k-1} d_\ell}{u(D)^{k-1}} = \frac{(-1)^{k-1} \beta_{1,k-1}(D)}{u(D)^{k-1}}
\end{equation}

\end{enumerate}
\end{definition}

Our upper bounds are obtained as a direct consequence of the following variational definition of $s_{min}(D+U_N)$.

\begin{theorem}\label{Variational} \
With the notations above, 
\begin{equation}
s_{\min}^2(D+U_N)= \min_{z \in \C^N} \frac{\sum_{k=1}^N|\delta_{(k+1)}|^2|u(D)z_k+z_{(k+1)}|^2}{\sum_k|\delta_k|^2|z_k|^2}
\end{equation}
\end{theorem}

\begin{proof}
The usual characterization of $s_{min}(D+U_N)$ is
\begin{equation}
s_{\min}^2(D+U_N) = \min_{x\in\C^N}\frac{\sum_{k=1}^N|d_kx_k+x_{(k+1)}|^2}{\sum_{k=1}^N|x_k|^2}
\end{equation}

Let $x\in\C^N$ and define $z\in\C^N$ by 
\begin{equation}
x_k=\delta_kz_k,
\end{equation}
which is possible since the $\delta_k$'s do not vanish. 
Then,
\begin{eqnarray}
d_kx_k+x_{(k+1)}&=&d_k\delta_kz_k+\delta_{(k+1)}z_{(k+1)}\\
&=&\delta_{(k+1)}(uz_k+z_{(k+1)}),
\end{eqnarray}

so that 
\begin{equation}
\sum_{k=1}^N|d_kx_k+x_{(k+1)}|^2=\sum_{k=1}^N|\delta_{(k+1)}|^2|uz_k+z_{(k+1)}|^2
\end{equation}

and
\begin{equation}
\sum_{k=1}^N|x_k|^2=\sum_{k=1}^N|\delta_k|^2|z_k|^2
\end{equation}

This shows that
\begin{equation}
s_{\min}^2(D+U_N)=\min_x\frac{\sum_{k=1}^N|d_kx_k+x_{(k+1)}|^2}{\sum_{k=1}^N|x_k|^2}
=\min_z\frac{\sum_{k=1}^N|\delta_{(k+1)}|^2|uz_k+z_{(k+1)}|^2}{\sum_k|\delta_k|^2|z_k|^2}
\end{equation}

\end{proof}

This variational characterization gives first the following weak bound, which compares the general case to the scalar case.

\begin{lemma}\label{Lemma-Comparison-to-scalar}\
\begin{enumerate}
\item
\begin{equation}
s_{\min}(D+U_N)\leq \varphi_N(u(D))
\end{equation}
\item
\begin{equation}
s_{\min}(A)\leq \min_{1 \leq i \leq K(\sigma)}   \varphi_{N_i}(u(D_i))
\end{equation}
\end{enumerate}
\end{lemma}

\begin{proof}
We bound $s_{\min}(D+U_N)$ using an eigenvector of $(u(D)Id_N+U_N)(u(D)Id_N+U_N)^\ast$ as a test vector $z\in\C^N$ in Theorem~\ref{Variational}, i.e. we use  $z^\omega$ defined by

\begin{equation}
z_k^\omega=\omega^k \qquad \omega\in\cU_N.
\end{equation}

By  Theorem~\ref{Variational}, we see that

\begin{eqnarray}
s_{\min}^2(D+U_N)&\leq&\frac{\sum_{k=1}^N|\delta_{(k+1)}|^2|uz_k^\omega+z_{(k+1)}^\omega|^2}{\sum_k|\delta_k|^2|z_k^\omega|^2}\\
&\leq&\frac{\sum_k|\delta_{(k+1)}|^2|u+\omega|^2}{\sum_k|\delta_k|^2}\\
&\leq&|u+\omega|^2
\end{eqnarray}

Minimizing over $\omega\in\cU_N$ we get
\begin{equation}
s_{\min}^2(D+U)\leq \min_{\omega} |u+\omega|^2=s_{\min}^2(uId_N+U_N),
\end{equation}
so that
\begin{equation}
s_{\min}(D+U_N)\leq \varphi_N(u),
\end{equation}
which is the announced upper bound in the first part of the Lemma.
The proof of the second part is then a direct consequence of Theorem~\ref{CycleDecomposition}
\end{proof}

It is possible to get a much sharper bound using a better choice of test vector in Theorem~\ref{Variational}.

\begin{theorem}\label{UBgamma}\
\begin{enumerate}
\item
\begin{equation}
s^2_{min}(D+U_N) \leq \frac{c_0(D)}{\gamma_N(D)}
\end{equation}
\item 
\begin{equation}
 s^2_{min}(A) \leq  \min_{1\leq i\leq K(\sigma)}\frac{c_0(D_i)}{\gamma_N(D_i)}
\end{equation}
\end{enumerate}
\end{theorem}

\begin{proof}
We prove the first part of the theorem. The second part is again a direct consequence of Theorem~\ref{CycleDecomposition}.
We make a better choice of test vector in Theorem~\ref{Variational}. 
More precisely, we fix $k_0\leq N$ and choose $z$ as follows:

\begin{eqnarray}
z_1&=&(-u)^{N-k_0+1}\\ z_2&=&(-u)^{N-k_0+2}\\
z_{k_0-1}&=&(-u)^{N-1} \\ z_{k_0}&=&1\\
z_{k_0+1}&=&-u\\ z_N&=&(-u)^{N-k_0}
\end{eqnarray}

We then have
\begin{equation}
uz_k+z_{(k+1)}=0
\end{equation}

for every $1\leq k\leq N$, but for $k=k_0-1$ we have

\begin{equation}
uz_k+z_{(k+1)}=(-1)^{N-1}u^N+1.
\end{equation}
By the definition of $\delta_k$ in Definition~\ref{delta}, we then have
\begin{eqnarray}
\sum_{k=1}^N|\delta_{(k+1)}|^2|uz_k+z_{(k+1)}|^2=c_0|\delta_{k_0}|^2= c_0\frac{\prod_{\ell=1}^{k_0-1}|d_\ell|^2}{|u|^{2(k_0-1)}}
\end{eqnarray}

Moreover,
\begin{equation}
\sum_{k=1}^N|\delta_k|^2|z_k|^2=\sum_{k=1}^{k_0-1}|\delta_k|^2|z_k|^2+\sum_{k=k_0}^N|\delta_k|^2|z_k|^2.
\end{equation}

The second sum can be written as
\begin{eqnarray}
\sum_{k=k_0}^N|\delta_k|^2|z_k|^2&=&|\delta_{k_0}|^2+|u|^2|\delta_{k_0+1}|^2+\dots +|u|^{2(N-k_0)}|\delta_N|^2\\
&=&\frac{1}{|u|^{2(k_0-1)}}\left(\prod_{\ell=1}^{k_0-1}|d_\ell|^2+\dots +\prod_{\ell=1}^{N-1}|d_\ell|^2\right)
\end{eqnarray}

The first sum is equal to
\begin{equation}
\sum_{k=1}^{k_0-1}|\delta_k|^2|z_k|^2=\frac{|u|^{2N}}{|u|^{2(k_0-1)}}\left(1+\prod_{\ell=1}^1|d_\ell|^2+\prod_{\ell=1}^2|d_\ell|^2+\dots+\prod_{\ell=1}^{k_0-2}|d_\ell|^2\right)
\end{equation}

Then, Theorem~\ref{Variational} shows that
\begin{eqnarray}
s_{\min}^2(D+U_N)&=&\frac{c_0\prod_{\ell=1}^{k_0-1}|d_\ell|^2}{\sum_{L=k_0-1}^{N-1}\prod_{\ell=1}^L|d_\ell|^2+\prod_{\ell=1}^N|d_\ell|^2\left(1+\sum_{L=1}^{k_0-2}\prod_{\ell=1}^L|d_\ell|^2\right)}\\
&\leq&\frac{c_0}{\sum_{L=k_0-1}^N\prod_{\ell=k_0}^L|d_\ell|^2}\\
\end{eqnarray}

We can now optimize in $k_0\in\{1, \dots, N\}$ so that
\begin{eqnarray}
s_{\min}^2(D+U_N)&\leq&\frac{c_0}{\max_{1\leq k_0\leq N}\sum_{k_0\leq L\leq N}\prod_{\ell=k_0}^L|d_\ell|^2}\\
&=&\frac {c_0(D)}{\gamma_N(D)}
\end{eqnarray}
which proves the upper bound for the smallest singular value of $D+U_N$.
\end{proof}

\section{ Proof of Theorem~\ref{OneSided}}\label{section_proof_onesided}

In this section we prove Theorem~\ref{OneSided}, using Corollary~\ref{OperatorNorm} and the bounds established in the two preceding sections.

We will prove below the following lower bound on $s_{min}(D+U_N)$.

\begin{theorem}\label{sharpboundOneSided}\
\begin{enumerate}
\item
Assume that $|d_\ell| <1$, for every $1 \leq \ell \leq N$. Define $\epsilon_N= 1-\max_{1 \leq \ell \leq N} |d_\ell| >0$. Then,
\begin{equation}
s_{min}(D+U_N) \geq \frac{1}{2\sqrt{2}} \epsilon_N
\end{equation}
\item
Assume that $|d_\ell| >1$, for every $1 \leq \ell \leq N$. Define $\epsilon_N=  1 - (\min_{1 \leq \ell \leq N} |d_\ell|)^{-1}  >0$. Then,
\begin{equation}
s_{min}(D+U_N) \geq \frac{1}{2\sqrt{2}} \epsilon_N
\end{equation}
\end{enumerate}
\end{theorem}

This result implies immediately Theorem~\ref{OneSided} since we know that
\begin{equation}
s_{min}(A)= \min_{1 \leq i \leq K(\sigma)} s_{min}(D_i+U_{N_i}) \geq \frac{1}{2\sqrt{2}} \min_{1 \leq i \leq K(\sigma)} \epsilon_{N_i} =  \frac{1}{2\sqrt{2}} \epsilon_N
\end{equation}

\begin{proof}[Proof of Theorem~\ref{sharpboundOneSided}]

We begin by considering the first case, where $|d_\ell| <1$ for every $1 \leq \ell \leq N$.

We have seen that by Corollary~\ref{OperatorNorm}, 
\begin{equation}
s_{min}(D+U_N) = || B(D) + C(D)||^{-1}
\end{equation}
Obviously,
\begin{equation}
|| C(D)+B(D)|| \leq ||C(D)|| + ||B(D)|| 
\end{equation}
It thus suffices to prove the following upper bounds for the operator norms of $ B(D)$ and of $C(D)$.

\begin{lemma}\label{OperatorNormsBD}\
\begin{enumerate}
\item The operator (or HS) norm of the matrix $C(D)$ is bounded above by
\begin{equation}
|| C(D) ||^2 \leq \frac{2}{\epsilon_N^2}
\end{equation}
\item The operator norm of the matrix $B(D)$ is bounded above by
\begin{equation}\label{BoundB(D)}
|| B(D) ||^2 \leq \frac{2}{\epsilon_N^2}
\end{equation}

\end{enumerate}
\end{lemma}
 
\begin{proof}[Proof of Lemma~\ref{OperatorNormsBD}]

Letting $r_N := \max_{1 \leq \ell \leq N} |d_{\ell}| < 1$, we note that
\begin{equation}\label{BoundBeta}
|\beta_{k,m}| \leq r_N^{(m-k+1)} \1_{k \leq m+1}
\end{equation}

We begin with the estimation of $||C(D)||$. We have seen that by Corollary~\ref{OperatorNorm}:
\begin{equation}
||C(D)||^2 = \rho^{(1)}_N(D)
\end{equation} 

Recall that
\begin{equation}
\rho^{(1)}_N(D) = ||C(D)||_{HS}^2= \frac{1}{c_0(D)} \sum_{k=1}^N \ |\beta_{1, k-1}|^2 \sum_{m=1}^N \ |\beta_{m+1,N}|^2
\end{equation}

We will need first the following trivial bound on $c_0(D)$
\begin{equation}
c_0(D) \geq (1 -r_N^N)^2 
\end{equation}
This bound is clear since 
\begin{equation}
c_0(D) = | \prod_{\ell=1}^N d_\ell - (-1)^N|^2
\end{equation}
And 
\begin{equation}
 | \prod_{\ell=1}^N d_\ell| \leq r_N^N
\end{equation}

So that, using the bound~\eqref{BoundBeta},
\begin{equation}
 \rho^{(1)}_N(D) \leq \frac{1}{(1-r_N^N)^2}  \sum_{k=1}^N r_N^{2(k-1)} \sum_{m=1}^N r_N^{2(N-m)}
\end{equation}
After an obvious reindexing,
\begin{equation}
 \rho^{(1)}_N(D) \leq \frac{1}{(1-r_N^N)^2} ( \sum_{k=0}^{N-1} r_N^{2k})^2 = \frac{1-r^{2N}_N}{(1-r_N^N)^2(1-r_N^2)}
\end{equation}
and finally
\begin{equation}
 \rho^{(1)}_N(D) \leq \frac{1+r^N_N}{(1-r_N^N)(1-r_N^2)}
\end{equation}
Since $0< r_N < 1$, it is clear that 
\begin{equation}
\frac{1+r^N_N}{(1-r_N^N)(1-r_N^2)} \leq \frac{2}{(1-r_N)^2} = \frac{2}{\epsilon_N^2}
\end{equation}

This proves the first item of our lemma.

We now estimate the operator norm $||B(D)|| = \max_{ y \in \C^N} \frac{||B(D)y||}{||y||}$.

Let $y \in \C^N$, then
\begin{equation}
||B(D)y||^2 = \sum_{k=1}^N \big|\sum_{m=1}^N \beta_{k+1,m-1} \1_{k \leq m-1} y_m\big|^2 \leq \sum_{k=1}^N \big(\sum_{m=1}^N |\beta_{k+1,m-1}| \1_{k \leq m-1} |y_m| \big)^2
\end{equation}

Expanding the square, and using the bound ~\eqref{BoundBeta}, we get

\begin{equation}
||B(D)y||^2 \leq \sum_{k=1}^N \sum_{m=1}^N \sum_{m'=1}^N \1_{k\leq m-1} \1_{k\leq m'-1}|\beta_{k+1,m-1}||\beta_{k+1,m'-1}|  |y_m| |y_{m'}|
\end{equation}

And thus
\begin{equation}
||B(D)y||^2 \leq \sum_{m=1}^N \sum_{m'=1}^N  a_{m,m'} |y_m| |y_{m'}|
\end{equation}
where
\begin{equation}
a_{m,m'} =  \sum_{k=1}^N \1_{k\leq m-1} \1_{k\leq m'-1}|\beta_{k+1,m-1}||\beta_{k+1,m'-1}|
\end{equation}

Using the bound ~\eqref{BoundBeta}, and assuming, wlog, that $m\leq m'$,we see that

\begin{equation}
a_{m,m'} \leq  \sum_{k=1}^N \1_{k\leq m-1} r_N^{m+m'-2k-2} = r_N^{m+m'-4} \sum_{l=0}^{m-2} r_N^{-2l} 
\end{equation}

So that

\begin{equation}
a_{m,m'} \leq  r_N^{m+m'-4} \frac{r_N^{-2m+4} -1}{r_N^{-2} -1} \leq \frac{ r_N^{m+m'-2m}}{r_N^{-2} - 1} \leq \frac{ 1}{\epsilon_N} r_N^{m'-m}
\end{equation}

Summarizing, 
\begin{equation}
||B(D)y||^2 \leq \frac{2}{\epsilon_N} \sum_{m=1}^N \sum_{m'=1}^N  r_N^{m'-m} \1_{m \leq m'}|y_m| |y_{m'}|
\end{equation}

It is now easy to conclude, using the following classical estimate.

\begin{lemma}\label{BoundToeplitz}
For any $r <1$, and any $a \in \C^N$
\begin{equation}
|\sum_{m=1}^N \sum_{m'=1}^N r^{m'-m} \1_{m \leq m'} a_m \overline{a_{m'}}| \leq \frac{1}{1-r} \sum_{m=1}^N |a_m|^2
\end{equation}
\end{lemma}

 Indeed we see then that 
 \begin{equation}
 ||B(D)y||^2 \leq \frac{2}{\epsilon_N^2} ||y||^2 
 \end{equation}
 which proves the bound ~\eqref{BoundB(D)} on the operator norm of $B(D)$.
 \end{proof}

For the sake of completeness, we provide here a proof of Lemma~\ref{BoundToeplitz}, which is a very simple case of classical bounds on operator norms of Toeplitz matrices.

 \begin{proof}[Proof of Lemma~\ref{BoundToeplitz}]
 Consider a doubly-infinite sequence $(t_m)_{m \in \Z}$ in $\ell^1(\Z)$.

Consider the function 
\begin{equation}
f(x) = \sum_{m=- \infty}^{\infty} t_m e^{imx} 
\end{equation}
Then

\begin{eqnarray}
\sum_{m=1}^N \sum_{m'=1}^N t_{m-m'} a_m \overline{a_{m'}} &=&\sum_{m=1}^N\sum_{m'=1}^N (\frac{1}{2\pi} \int_0^{2\pi} f(x) e^{-i(m-m')x} dx) a_m \overline{a_{m'}}\\
&=& \frac{1}{2\pi} \int_0^{2\pi} f(x) |\sum_{m=1}^N a_me^{-imx}|^2 dx
\end{eqnarray}

Using this identity, with the sequence $t_m = \1_{m=0}$, i.e for the constant function $ f(x)=1$, one sees that
\begin{equation}
 \frac{1}{2\pi} \int_0^{2\pi} |\sum_{m=1}^N a_me^{-imx}|^2 dx = \sum_{m=1}^N |a_m|^2
\end{equation}

So that, for a general sequence $(t_m)_{m \in \Z}$ in $\ell^1(\Z)$, we have
\begin{eqnarray}
|\sum_{m=1}^N \sum_{m'=1}^N t_{m-m'} a_m \overline{a_{m'}}| &\leq& \frac{1}{2\pi} \int_0^{2\pi} |f(x)| \Big|\sum_{m=1}^N a_me^{-imx}\Big|^2 dx\\
&\leq&||f||_{\infty} \frac{1}{2\pi} \int_0^{2\pi} \Big|\sum_{m=1}^N a_me^{-imx}\Big|^2 dx\\
&=&  ||f||_{\infty} \sum_{m=1}^N |a_m|^2
\end{eqnarray}

In order to prove our Lemma, it suffices to choose $t_m = r^m \1_{m \geq 0}$.
Then $f(x) = \frac{1}{1- re^{ix}}$, and
\begin{equation}
||f||_\infty = |f(0)| = \frac{1}{1-r}
\end{equation} 

\end{proof}

We still have to prove the second item of Theorem~\ref{sharpboundOneSided}, i.e the case where $|d_\ell| >1$ for every $1 \leq \ell \leq N$.
This second item is a direct consequence of the first, and of the following simple "duality" result.

\begin{lemma}\label{duality}
If $\forall \ell$, $|d_\ell|\geq \gamma>0$, then
\begin{equation}
s_{\min}(D+U_N) \geq \gamma\cdot s_{\min}(\hat{D}^{*}+U_N),
\end{equation}
where
\begin{equation}
\hat{D}=\text{diag}\left(\frac 1{d_N}, \frac 1{d_1}, \frac 1{d_2}, \dots, \frac 1{d_{N-1}}\right).
\end{equation}
\end{lemma}

\begin{proof}[Proof of Lemma~\ref{duality}]
The matrix $U_N^{-1}$ is associated to the cycle $(N \ N-1 \ \dots \ 1)$, i.e. 
\begin{equation}
U_N^{-1}=\begin{pmatrix}0 & & & 1\\ 1 & \ddots & & \\ & \ddots & \ddots & \\ 0 & & 1 & 0\end{pmatrix},
\end{equation}
so that 
\begin{equation}
\hat{D}+U_N^{-1}=\begin{pmatrix}\frac 1{d_N} & & & 1\\ 1 & \frac 1{d_1} & & \\ & \ddots & \ddots & \\ 0 & & 1 & \frac 1{d_{N-1}}\end{pmatrix}
\end{equation}
and

\begin{equation}
s_{\min}^2(\hat{D}+U_N^{-1})=\min_x\frac{\sum_k|x_k+\frac 1{d_k}x_{(k+1)}|^2}{\sum_k|x_k|^2}.
\end{equation}

But
\begin{eqnarray}
s_{\min}^2(D+U_N)&=&\min_x\frac{\sum_k|d_kx_k+x_{(k+1)}|^2}{\sum_k|x_k|^2}\\
&=& \min_x\frac{\sum_k|d_k|^2|x_k+\frac 1{d_k}x_{(k+1)}|^2}{\sum_k|x_k|^2},
\end{eqnarray}
so that

\begin{equation}
s_{\min}^2(D+U_N)\geq \gamma^2 s_{\min}^2(\hat{D}+U_N^{-1}).
\end{equation}
But, since $U_N$ is unitary, then $\hat{D}+U_N^{-1} = \hat{D}+U_N^{*}= (\hat{D}^{*}+U_N)^{*}$.
The singular values of $\hat{D}+U_N^{-1}$ are thus  the same as the singular values of $\hat{D^{*}}+U_N$. So  that
\begin{equation}
s_{\min}^2(\hat{D}+U_N^{-1}) = s_{\min}^2(\hat{D}^{*}+U_N)
\end{equation}
This concludes the proof of Lemma~\ref{duality}.
\end{proof}

The moduli  $ \frac{1}{|d_i|}$of the diagonal elements of $\hat{D}^{*}$ are all  smaller than $1$, 
we can then apply the first item of Theorem~\ref{sharpboundOneSided}, and conclude that, with
$\epsilon_N=  1 - (\min_{1 \leq \ell \leq N} |d_\ell|)^{-1}  >0$, we have
\begin{equation}
s_{\min}(D+U_N) = s_{\min}^2(\hat{D}^{*}+U_N) \geq \frac{1}{2\sqrt{2}} \epsilon_N
\end{equation}

This concludes the proof of Theorem~\ref{sharpboundOneSided}

\end{proof}

\section{The case of i.i.d. entries}\label{sectioniid} 
\subsection{Proof of Theorem~\ref{NegativeDrift1}}\label{negd1}

In this chapter, we give a proof of Theorem~\ref{NegativeDrift1} and Theorem~\ref{NegativeDrift2}.
Our main tool will be the classical theory of excursions and fluctuations for $1$-d random walks.
We will have to go a bit further than the classical theory (see references in the Appendix).

Define $\xi_i=2\log|d_i|$ and

\begin{equation}
S_n=\sum_{i=1}^n\xi_i=\log\prod_{i=1}^n|d_i|^2.
\end{equation}

$S_n$ is a random walk with negative drift.

It is related to our problem by the obvious formula, for $1 \leq k \leq m$
\begin{equation}
|\beta_{k,m}|^2 = \exp{ (S_m - S_{k-1})}
\end{equation}
which is also valid for $k= m+1$, since by definition $ \beta_{m+1,m} = 1$.\\

We will now show that the lower and upper bounds we have found for $s_{min}(D+U_N)$ can easily be controlled in terms of functionals of the random walk $S_n$. Indeed, our upper bound estimates are in terms of the quantities $c_0(D)$, $\gamma_N(D)$ (see Definition~\ref{definitionrho}). Our lower bound estimates are given in terms of the quantities $\rho^{(1)}_N(D)$ and $\rho^{(2)}_N(D)$ (see Definition~\ref{definitionrho}).  We estimate these quantities in terms of functionals of the random walk $S_n$ in the next lemma. We first introduce the relevant functionals of the random walk $S_n$.

\begin{definition}\label{functionalsRW}
We define the following important functionals of the random walk $S_n$.

\begin{enumerate}
\item
\begin{equation}
U_\infty=\sum_{k=1}^\infty e^{S_k}
\end{equation}
\item
\begin{equation}
M_N=\max_{1\leq k \leq  m \leq N}(S_m-S_k).
\end{equation}
\item
\begin{equation}
T_N=\sum_{1\leq k \leq m \leq N}e^{S_m-S_k}.
\end{equation}
\end{enumerate} 
\end{definition}

We now restate the bounds obtained in Section ~\ref{LowerBounds} and Section~\ref{UpperBounds} in terms of these functionals.
\begin{theorem} \label{ThmboundsRW}
With the notations above
\begin{equation}
(2c_0(D)^{-1}(1+U_\infty)(1+\hat U_\infty) + 2T_N)^{-1} \leq s_{min}^2(D+U_N) \leq c_0(D) e^{-M_N}
\end{equation}
Where $\hat U_\infty$ is a random variable with the same distribution as $U_\infty$.

\end{theorem}

\begin{remark}
The behavior the distributions of the functionals $U_\infty$ and $M_N$ has been fully understood and is a central topic of the classical fluctuation theory of random walks (see the Appendix for references and relevant statements). The behavior of the functional $T_N$ has not been studied, and we will have to derive it in the Appendix. We lose a logarithmic term there. It is plausible that one could get rid of this logarithmic correction, extending some of the best tools available in the classical theory, but the needed effort might be sizable.
\end{remark}

\begin{proof}[Proof of Theorem~\ref{ThmboundsRW}]
The proof of this theorem is a consequence of the following lemma.

\begin{lemma}\label{controlRW}
With the notations above, we have
\begin{enumerate}

\item
\begin{equation}
\gamma_N \geq \exp M_N,
\end{equation}

\item
\begin{equation}
\rho^{(1)}_N(D) \leq \frac{1}{c_0(D)} (1+U_\infty)(1+ \hat U_\infty),
\end{equation}
where $U_\infty$ and $\hat U_\infty$ are two random variable distributed as $U_\infty$ (but not independent).
\item
\begin{equation}
\rho^{(2)}_N(D) \leq T_N
\end{equation}

\end{enumerate}
\end{lemma}

\begin{proof}[Proof of Lemma~\ref{controlRW}]

%Obviously,
%\begin{equation}
%c_0(D) = | (-1)^N - \prod_{\ell=1}^N d_{\ell}|^2 \leq (1+ |\prod_{\ell=1}^N d_{\ell}|)^2
%\end{equation}
%and 
%\begin{equation}
 %| \prod_{\ell=1}^N d_{\ell}| = \exp{ \frac{1}{2} S_N}
%\end{equation}
%which proves the first item.

We recall that
\begin{equation}
\gamma_N(D) = \max_{1\leq k \leq N} \sum_{m=1}^N |\beta_{k,m}|^2 \geq  \max_{1\leq k \leq N} \max_{1\leq m \leq N} |\beta_{k,m}|^2.
\end{equation}
Thus,
\begin{equation}
\gamma_N \geq  \exp{M_N},
\end{equation}
which proves the first item.

To prove the second item, recall that
\begin{equation}
\rho^{(1)}_N(D) =  \frac{1}{c_0(D)} \sum_{k=1}^N \ |\beta_{1, k-1}|^2 \sum_{m=1}^N \ |\beta_{m+1,N}|^2
\end{equation}

Define the two random variables 
\begin{equation}
U_N= \sum_{\ell=1}^{N-1} e^{S_{\ell}} 
\end{equation}
and
\begin{equation}
\hat U_N= \sum_{\ell=1}^{N-1} e^{S_N - S_{\ell} }
\end{equation}

It is obvious that

\begin{equation}
\sum_{k=1}^N \ |\beta_{1, k-1}|^2 = \sum_{k=1}^N e^{S_{k-1}} = 1 + \sum_{\ell=1}^{N-1} e^{S_{\ell}} \leq 1 + U_N
\end{equation}

Similarly

\begin{equation}
\sum_{m=1}^N \ |\beta_{m+1, N}|^2 = 1 + \sum_{\ell=1}^{N-1} e^{S_N - S_{\ell} } \leq 1 + \hat U_N
\end{equation}

So that 

\begin{equation}
\rho^{(1)}_N(D) =  \frac{1}{c_0(D)} ( 1 + U_N)(1+\hat U_N)
\end{equation}

Consider the sequence of random variables $\xi = (\xi_k)_{ k\geq 1}$, and the sequence of random variables $\hat \xi^N = (\hat \xi^N_k)_{ k\geq 1}$, obtained by time-inversion at $N$, i.e. defined, by $\hat \xi^N_k = \xi_{N-k+1}$ for $ 1 \leq k \leq N$, and by $\hat \xi^N_k = \xi_k$ for $N+1\leq k$. These two infinite sequences $\xi$ and $\hat \xi^N$ obviously have the same distribution. So that the two random walks $(S_n)_{n \geq 1}$ and $(\hat S_n)_{n \geq 1}$, defined as their partial sums $S_n = \sum_{k=1}^n \xi_k$ and $\hat S_n = \sum_{k=1}^n \hat \xi_k$, also have the same distribution.

Noting that
\begin{equation}
\hat U_N =  \sum_{\ell=1}^{N-1} e^{\hat S_{\ell}},
\end{equation}
it is clear that the two random variables $U_N$ and $\hat U_N$ have the same distribution. 
Moreover, it is clear that $ U_N \leq U_\infty=\sum_{k=1}^\infty e^{S_k}$ and $ \hat U_N \leq \hat U_\infty=\sum_{k=1}^\infty e^{\hat S_k}$, where these two random variables $ U_\infty$ and $\hat U_\infty$ have the same distributions. Thus,
\begin{equation}
\rho^{(1)}_N(D) \leq  \frac{1}{c_0(D)} ( 1 + U_\infty)(1+\hat U_\infty)
\end{equation}

This proves our second item. 

We now prove the last item.
Recall that

\begin{equation}
\rho^{(2)}_N(D) = ||B(D)||_{HS}^2= \sum_{1\leq k\leq m-1 \leq N} |\beta_{k+1,m-1}|^2
\end{equation}

Obviously,

\begin{equation}
\rho^{(2)}_N(D)  \leq \sum_{1\leq k\leq m-1 \leq N} e^{(S_{m-1} - S_k)} \leq T_N
\end{equation}

which proves our last item.
\end{proof}

By Theorem~\ref{LBrho} and Theorem~\ref{UBgamma}, we have now proved the following bounds for $s_{min}(D+U_N)$

\begin{equation}\label{boundsRW}
(2 c_0(D)^{-1}(1+U_\infty)(1+\hat U_\infty) + 2T_N)^{-1} \leq s_{min}^2(D+U_N) \leq c_0(D) e^{-M_N}
\end{equation}

and thus proved Theorem~\ref{ThmboundsRW}.
\end{proof}

We now prove Theorem~\ref{NegativeDrift1} and Theorem~\ref{NegativeDrift2}
\begin{proof}[Proof of Theorem~\ref{NegativeDrift1}]

Let us define the random variable $X_N$ by
\begin{equation}
X_N= c_0(D)^{-1}(1+U_\infty)(1+\hat U_\infty)
\end{equation}

We want to estimate the lower tail of $s_{min}(A)$ for a general permutation $\sigma$. We will of course begin by the analogous result for the case of a single-cycle permutation, i.e. an estimate for the lower tail of $s_{min}(D+U_N)$.

From Theorem~\ref{ThmboundsRW}, we know that

\begin{equation}\label{boundsRW2}
(2X_N + 2T_N)^{-1} \leq s_{min}^2(D+U_N)
\end{equation}

When $\theta <1$, the following bound on the tail of $T_N$ is given in Theorem~\ref{T_N2}, for any $u \leq cN^{\frac{1}{\theta}}$
\begin{equation}
\bbP [ T_N  \geq u] \leq CN \frac{(\ln u)^{\theta}}{u^{\theta}}     
\end{equation}

The following tail estimate is also given in the Appendix, see Lemma~\ref{X_N},
\begin{equation}
\bbP[ X_N \geq u] \leq \frac{C}{u^{\delta}}   e^{-\frac{\delta}{\gamma}kN}
\end{equation}

So that, a simple union bound yields

\begin{equation}
\bbP[s_{min}^2(D+U_N) \leq t] \leq \bbP[X_N \geq \frac{1}{4t}] + \bbP[T_N  \geq \frac{1}{4t}]
\end{equation}
and 
\begin{equation}
\bbP[s_{min}^2(D+U_N) \leq t] \leq Ct^{\delta} e^{-\frac{\delta}{\gamma}kN}+ CN t^{\theta} (\ln \frac{1}{t})^{\theta} 
\end{equation}

We have thus proved the first item of Theorem~\ref{NegativeDrift2} in the case of a single-cycle permutation.
In order to get to the case of a general permutation matrix, now we simply use again the fact that 

\begin{equation}
s_{min}^2(A) = \min_{1 \leq i \leq K(\sigma)} s_{min}^2(D_i + U_{N_i})
\end{equation}

and a union bound to obtain

\begin{equation}
\bbP[s_{min}^2(A) \leq t] \leq \sum_{1 \leq i \leq K(\sigma)} \bbP[s_{min}^2(D_i+U_{N_i}) \leq t]
\end{equation}

Now, by the estimate above for the single-cycle case, we get that

\begin{equation}
\bbP[s_{min}^2(A) \leq t] \leq Ct^{\delta} \sum_{1 \leq i \leq K(\sigma)} e^{-\frac{\delta}{\gamma}kN_i} + C t^{\theta} (\ln \frac{1}{t})^{\theta}\sum_{1 \leq i \leq K(\sigma)}N_i 
\end{equation}

This proves the first item of Theorem~\ref{NegativeDrift2}, for a general permutation.
The second item is then a very simple consequence. Indeed, choosing $t = \frac{u}{N^{\frac{1}{\theta}}}$ in this estimate, yields
\begin{equation}
\limsup_{ N \to \infty} \bbP_D [ s^2_{min}(D+U_N) \leq \frac{u}{N^{\frac{1}{\theta}}}] \leq C u^{\theta}(\ln \frac{1}{u})^{\theta}
\end{equation}
The sequence of distributions of the random variables $ (N^{\frac{1}{\theta}} s^2_{min}(D+U_N))^{-1} $ is thus tight.
This concludes the proof of Theorem~\ref{NegativeDrift1}.

\end{proof}

We now prove Theorem~\ref{NegativeDrift2}
\begin{proof}[Proof of Theorem~\ref{NegativeDrift2}]
Here we assume $\theta >1$. The proof is totally parallel to the one we just gave, using the following tail estimate for $T_N$.
For $ u \geq C N \log N$, the bound from Theorem~\ref{T_N2} shows that

\begin{equation}
\bbP [ T_N  \geq u] \leq CN \frac{(\ln u)^{\theta}}{u^{\theta}}     
\end{equation}

we see that, for $ u \geq C N \log N$,

\begin{equation}
\bbP[s_{min}^2(D+U_N) \leq t] \leq C t^{\gamma}+ CN t^{\theta} (\ln \frac{1}{t})^{\theta} 
\end{equation}

Thus we have proved the first statement of Theorem~\ref{NegativeDrift2}.

Choosing $t= \frac{u}{N \ln N}$ we have

\begin{equation}
\limsup_{ N \to \infty} \bbP_D [ s^2_{min}(D+U_N) \leq \frac{u}{N \ln N}] \leq C \frac{u^{\theta}}{N^{\theta-1}}
\end{equation}
This shows that the sequence of random variables $ (s^2_{min}(D+U_N)N \ln N)^{-1}$ converges to zero in probability (or even a.s if $\theta >2$). This concludes the proof  of Theorem~\ref{NegativeDrift3}.
\end{proof}

\subsection{Proof of Theorem~\ref{NegativeDrift3}}\label{negd3}
We can now prove Theorem~\ref{NegativeDrift3}, i.e. give upper bounds for $s_{min}(A)$.

\begin{proof}[Proof of Theorem~\ref{NegativeDrift3}]

The proof of this theorem is a consequence of Theorem~\ref{boundsRW} and classical results about the asymptotic behavior of the random variable $M_N$, recalled in the Appendix.
We begin, naturally, with the case of a single-cycle permutation ad give upper-tail estimates for $s^2_{min}(D+U_N)$.
\begin{theorem}\label{uppertailSN}
There exists constants $k>0$ and $C>0$, such that, for every $t>0$
\begin{equation}
\bbP_D[ s_{min}(D+U_N) \leq t] \leq e^{-kN} + e^{-CNt^{\theta}}
\end{equation}
\end{theorem}

\begin{proof}

We have seen that
\begin{equation}
s_{min}(D+U_N) \leq c_0(D) e^{-M_N}
\end{equation}

By a simple union bound
\begin{equation}
\bbP_D[ s_{min}(D+U_N) \leq 2t] \leq \bbP_D[ c_0(D)  \geq 2] + \bbP_D[ M_N \leq -\log t]
\end{equation}

A trivial large deviation bound shows that
\begin{equation}
\bbP_D[ c_0(D)  \geq 2] \leq e^{-k_1N}
\end{equation}
But, using the estimate~\ref{boundMN}, we see that

\begin{equation}
\bbP_D[ s_{min}(D+U_N) \leq 2t] \leq e^{-k_1N} + e^{-k_2N} + e^{-CNt^{\theta}}
\end{equation}
and thus if $0<t<T$
\begin{equation}
\bbP_D[ s_{min}(D+U_N) \leq 2t] \leq C(T) e^{-CNt^{\theta}}
\end{equation}
Which proves the first item of Theorem~\ref{NegativeDrift3} in the single-cycle case.
But obviously, for a general permutation, we have
\begin{equation}
\bbP_D[s^2_{min}(A) \leq 2t] \leq \prod_{i=1}^{K(\sigma)} \bbP_D[ s_{min}(D_i+U_{N_i}) \leq 2t] \leq C(T) e^{-C\sum_{i=1}^{K(\sigma})N_i t^{\theta}}
=C(T) e^{-CN t^{\theta}}
\end{equation}
which proves the first item of the Theorem for the general case.
The second item is a direct consequence of the first.
\end{proof}

\end{proof}

\appendix
\section{Random Walks with negative drifts}\label{appendix}
Let $(\xi_i)_{i\geq1}$ be a sequence of i.i.d. random variables, with common distribution $\nu$.
We will assume that 
\begin{enumerate}
\item[C1.]
$\nu$ has exponential moments, i.e. exists a $B>0$ such that, $\E{e^{t\xi}}<\infty$ for every $ t \in [0,B]$,
\item[C2.]
$\E{\xi}=m<0$ and
\item[C3.]
$\nu$ is non-lattice.
\end{enumerate} 

We will denote by $\theta$ the unique positive number $\theta$ such that $\E{e^{\theta \xi}}= 1$.

Let 
\begin{equation}
S_N=\sum_{i=1}^N\xi_i, \quad S_0=0.
\end{equation}

We are interested in the asymptotic behavior of the two random variables

\begin{equation}
M_N=\max_{1\leq k \leq  m \leq N}(S_m-S_k).
\end{equation}
 and 
\begin{equation}
T_N=\sum_{1\leq k \leq m \leq N}e^{S_m-S_k}.
\end{equation}

The asymptotic behavior of $M_N$ is well understood (see \cite{karlin_dembo}, Theorem A, p. 115, for discussion, references and an extension to the lattice case, as well as an extension to Markov Chains).

\begin{theorem}\label{M_N} 
Under the assumptions C1-C3 above, as $N$ goes to infinity,
\begin{enumerate}
\item  $\frac{M_N}{\log N}$ converges almost surely to $\frac{1}{\theta}$ and
\item  $M_N-\frac{\log N}{\theta} $ converges in distribution to a Gumbel variable
\begin{equation}
\lim_{N\rightarrow\infty}\bbP[M_N-\frac{\log N}{\theta}\leq x]=\exp(-C e^{-\theta x})
\end{equation}
\end{enumerate}
\end{theorem}

\begin{remark}
The value of the constant $C = C(\nu)$ is complicated as a function of the distribution $\nu$, but it is discussed in \cite{karlin_dembo}. 
\end{remark}

We need a more uniform estimate; There exists a constant $k_2>0$, such that
\begin{equation}\label{boundMN}
\bbP[ M_N \leq v] \leq e^{-k_2N} + e^{-NCe^{-\theta v}}
\end{equation}
This is a direct consequence (see \cite{iglehart1972}) of the fact that 
\begin{equation}
M_N  \leq \max_{1\leq k \leq i(N)} V_k
\end{equation}
where the random variables $V_k$ are i.i.d and have the same law as the variable $V$ defined by

\begin{equation}
V=\max_{1\leq n\leq K_1}S_n.
\end{equation}  

We will now prove an asymptotic estimate for the tail of the random variable $T_N$.

\begin{theorem}\label{T_N2} 
Under the assumptions C1-C3 above
\begin{enumerate}

\item If $\theta < 1$, then
\begin{equation}
\bbP [ T_N  \geq u] \leq CN \frac{(\ln u)^{\theta}}{u^{\theta}}     
\end{equation}
As a consequence the random variable $\frac{T_N}{N^{\frac{1}{\theta}}\ln N}$ is bounded in probability.
\begin{equation}
\bbP\Big[\frac{T_N}{N^{\frac{1}{\theta}}\ln N}  \geq v\Big] \leq \frac{C}{v^{\theta}}
\end{equation}

\item If $\theta >1 $, then there exists a $C_0>0$ such that, for $u > C_0 N \ln N$,

\begin{equation}
\bbP[T_N \geq u] \leq \frac{CN (\ln u)^{\theta}}{u^{\theta}}
\end{equation}
As a consequence the random variable $\frac{T_N}{N\ln N}$ converges to zero in probability.
\begin{equation}
\bbP\Big[\frac{T_N}{N\ln N} \geq v\Big] \leq \frac{C}{N^{\theta -1} v^{\theta}}
\end{equation}

\end{enumerate}
\end{theorem}

The proof of Theorem~\ref{T_N2} is rather involved. 

The first step is to introduce the following classical excursion decomposition of the path of the random walk $S_n$.
For any $c \geq 0$, consider the ladder epochs $(K_i(c))_{i\geq1}$ defined by 
\begin{eqnarray}
K_1&=&\min(n\geq1, S_n\leq -c)\\
K_2&=&\min(n\geq K_1+1, S_n-S_{K_1}\leq -c)\\
K_i&=&\min(n\geq K_{i-1}+1, S_n-S_{K_{i-1}}\leq -c)
\end{eqnarray}

And let 
\begin{equation}
U_i=\sum_{K_{i-1} \leq \ell<K_i}e^{S_\ell-S_{K_{i-1}}}.
\end{equation}
as well as the maximal length of the first $m$ excursions
\begin{equation}
R_m=\max_{1\leq i\leq m}(K_i-K_{i-1}).
\end{equation}
Obviously the random variables $U_i$ are i.i.d. (we recall that their common distribution depends on the parameter $c>0$)
We will first bound $T_N$ by a sum of these i.i.d random variables.

Let us denote by $i(\ell)$ the index of the excursion straddling $\ell$, for $1\leq\ell\leq N$, i.e.
\begin{equation}
K_{i(\ell)-1}\leq\ell < K_{i(\ell)}
\end{equation}
with the convention $K_0=0$.
Obviously $ i(\ell) \leq \ell$.

Then, we have the following upper bound for the random variable $T_N$. This is certainly a sub-optimal bound, where we lose a logarithmic term. But improving on this bound would be really too heavy here.

\begin{lemma}\label{boundT_N}
For $c>0$, there exists a constant $K(c)$ such that
\begin{equation}
T_N\leq K(c) R_{i(N)}\sum_{i=1}^{i(N)}U_i 
\end{equation}
\end{lemma}

\begin{proof}[Proof of Lemma~\ref{boundT_N}]

Let us first fix an integer $1 \leq \ell \leq N$.
For any $k\leq\ell$, such that $i(k)<i(\ell)$ we write
\begin{eqnarray}
S_\ell-S_k&=&S_\ell-S_{K_{i(\ell)-1}}+S_{K_{i(\ell)-1}}-S_{K_{i(k)}}+S_{K_{i(k)}}-S_k\\
&\leq&S_\ell-S_{K_{i(\ell)-1}}-c(i(\ell)-i(k)),
\end{eqnarray}

so that 
\begin{eqnarray}
\sum_{\substack{k\leq \ell\\ i(k)<i(\ell)}}e^{S_\ell-S_k}&\leq&\left[\sum_{1\leq i\leq i(\ell)-1}e^{-c(i(\ell)-i)}(K_{i+1}-K_i)\right]e^{S_\ell-S_{K_{i(\ell)-1}}}\\
&\leq&R_{i(\ell)}(\sum_{j=1}^\infty e^{-cj}) e^{S_\ell-S_{K_{i(\ell)-1}}}\\
&\leq&\frac{e^{-c}}{1-e^{-c}}R_{i(\ell)}e^{S_\ell-S_{K_{i(\ell)-1}}}
\end{eqnarray}

For $k\leq\ell$, such that $i(k)=i(\ell)$, we write
\begin{equation}
S_k\geq S_{K_{i(\ell)-1}}-c,
\end{equation}
so that
\begin{equation}
S_\ell-S_k\leq S_\ell-S_{K_{i(\ell)-1}}+c
\end{equation}

and
\begin{eqnarray}
\sum_{\substack{k\leq \ell\\ i(k)=i(\ell)}}e^{S_\ell-S_k}&\leq& e^{c}(K_{i(\ell)}-K_{i(\ell)-1})e^{S_\ell-S_{K_{i(\ell)-1}}}\\
&\leq&e^{c}R_{i(\ell)}e^{S_\ell-S_{K_{i(\ell)-1}}}.
\end{eqnarray}

Finally, we get
\begin{equation}
\sum_{k=1}^{\ell}e^{S_\ell-S_k}\leq K(c)R_{i(\ell)}e^{S_\ell-S_{K_{i(\ell)-1}}}
\end{equation}

where
\begin{equation}
K(c) = \frac{e^{-c}}{1-e^{-c}}+e^c
\end{equation}

and
\begin{equation}
T_N=\sum_{\ell=1}^N\sum_{k=1}^{\ell}e^{S_\ell-S_k}\leq K(c) R_{i(N)}\sum_{\ell=1}^Ne^{S_\ell-S_{K_{i(\ell)-1}}}.
\end{equation}

But, using the definition of the random variables $U_i$, we have
\begin{equation}
\sum_{\ell=1}^Ne^{S_\ell-S_{K_{i(\ell)-1}}} \leq \sum_{i=1}^{i(N)}U_i
\end{equation}

and therefore, 
\begin{equation}
T_N\leq K(c)R_{i(N)}\sum_{i=1}^{i(N)}U_i
\end{equation}
which is the bound we needed to prove.
\end{proof}

Lemma~\ref{boundT_N} shows that, in order to control the tail of the random variable $T_N$, it is sufficient to control the tail of the random variable $R_m$ and of the sum of i.i.d random variables $U_i$. We begin with the tail of the random variable $R_N$ and then turn to tail of the distribution of the sum of the $U_i$.

\begin{lemma}\label{R_N}
There exists a $\lambda >1$, such that, for any $y >  1+\lambda^{\frac{2}{3}}$
\begin{equation}
\bbP[ R_N \geq y] \leq C \frac{N}{\lambda^y}
\end{equation}
\end{lemma}

\begin{proof}[Proof of Lemma~\ref{R_N}]

Using the fact that the random variables $ (K_{i+1} - K_{i})$ are i.i.d, a trivial union bound shows that
\begin{equation}
\bbP[ R_N \geq y] \leq N \bbP[ K_1 \geq y] 
\end{equation}

Lemma~\ref{R_N} is thus a direct consequence of the following tail estimate for the random variable $K_1$ (the first ladder epoch).

\begin{lemma}\label{K1-tail}
There exists a $\lambda>1$ depending only on the distribution $\nu$, such that for any $ c \geq 0$ there exists a constant $C_1(c)>0$ with
\begin{equation}
\bbP[K_1\geq n]\sim \frac{C_1(c)}{\lambda^nn^{3/2}}
\end{equation}
\end{lemma}

Lemma~\ref{K1-tail} is a consequence of Theorem II, p. 241 in \cite{doney}, and Theorem 2.1 in \cite{iglehart}.
\end{proof}

We now turn to the tail of the distribution of the sum of i.i.d random variables $U_i$.
We need first to understand the tail behavior of the common distribution of the i.i.d  random variables $U_i$.

\begin{lemma}\label{Upper-Lower-bound-U}\
\begin{enumerate}
\item
There exist two positive constants $C$ and $C'$, such that, as $t\rightarrow\infty$, 
\begin{equation}
\frac{C}{t^{\theta}}\leq \bbP[U>t]\leq \frac{C'}{t^\theta}.
\end{equation}
\item
The expectation $\E U$ is finite iff $\theta >1$.
\end{enumerate}
\end{lemma}

\begin{proof}
Obviously, the second statement of Lemma~\ref{Upper-Lower-bound-U} is a direct consequence of the first.
We begin by proving the upper bound in this first statement. Clearly,
\begin{equation}
\bbP[U > t] \leq \bbP[ U_\infty > t],
\end{equation}
where
\begin{equation}
U_\infty=\sum_{n=1}^\infty e^{S_n}.
\end{equation}

The exact asymptotic behavior of the tail of $U_{\infty}$ is a simple consequence of known results.

\begin{lemma}\label{tailUinfty}
There exists a constant $C_5$ such that, as $t \rightarrow \infty$,
\begin{equation}
\bbP[ U_\infty > t] \sim C_5 t^{-\theta}
\end{equation}
\end{lemma}

\begin{proof}[Proof of Lemma~\ref{tailUinfty}]
Obviously,
\begin{equation}
U_\infty=e^{S_1}\sum_{n=2}^\infty e^{S_n-S_1}+e^{S_1},
\end{equation}

so that 
\begin{equation}
U_\infty\stackrel{(d)}{=}e^{\xi}U'_\infty +e^{\xi},
\end{equation}

where $U'_\infty$ is a random variable with the same distribution as $U_\infty$ and independent from $\xi$. This is an implicit renewal equation of the type treated by Kesten in \cite{kesten} and Goldie in \cite{goldie}. Theorem 4.1, p. 135 in \cite{goldie} implies then the tail estimate given in the lemma.
\end{proof}

So that we have proved the upper bound stated in Lemma~\ref{Upper-Lower-bound-U}.
We now turn to the lower bound. Recall that
\begin{equation}
V=\max_{1\leq n\leq K_1}S_n,
\end{equation}

then clearly,
\begin{equation}
U \geq e^V.
\end{equation}

The asymptotic behavior of the tail of $V$ is also known.

\begin{lemma}\label{boundV}
\begin{equation}
\bbP[V > y] \sim C_2(c)e^{-\theta y}\quad\text{as }y\rightarrow\infty
\end{equation}
\end{lemma}
This is proved in \cite{iglehart1972}, Theorem 1, p. 630 (or see \cite{karlin_dembo}, p. 115).

This obviously implies the lower bound of Lemma~\ref{Upper-Lower-bound-U}.

\end{proof}

The upper bound given in Lemma~\ref{Upper-Lower-bound-U} implies the following strong uniform bounds for the tail of the law of the random variable $\sum_{i=1}^N U_i$.

\begin{lemma}\label{Borovkov}\
\begin{enumerate}
\item
If $\theta < 1$, for any $T>1$, there exists a constant $C(T)>0$ such that, uniformly in $x \geq T N^{\frac{1}{\theta}}$
\begin{equation}
\bbP\Big[ \sum_{i=1}^N U_i \geq  x\Big] \leq \frac{C(T)N}{x^{\theta}}
\end{equation}
\item
If $1<\theta < 2$, for any $T>1$, there exists a constant $C(T)>0$ such that, uniformly in $x \geq T N^{\frac{1}{\theta}}$
\begin{equation}
\bbP \Big[ \sum_{i=1}^N U_i \geq N\E U + x\Big] \leq \frac{C(T)N}{x^{\theta}}
\end{equation}
\item
If $\theta>2$, for any $T>1$, there exists a constant $C(T)>0$ such that, uniformly in $x \geq T \sqrt{N \ln N}$
\begin{equation}
\bbP \Big[ \sum_{i=1}^N U_i \geq N\E U + x\Big] \leq \frac{C(T)N}{x^{\theta}}
\end{equation}
\end{enumerate}
\end{lemma}

This lemma is a consequence of the upper bound
\begin{equation}
\bbP[ U >t] \leq \frac{C}{t^{\theta}}
\end{equation}
 and of Lemma 2.1, Corollary 3.1 and Corollary 4.2 in \cite{borovkov}.
 
We are now, at long last, able to prove theorem~\ref{T_N2}.

\begin{proof}[Proof of Theorem~\ref{T_N2}]

For any $u>0$ and $s>0$,
\begin{equation}\label{T-R-U}
\bbP[T_N \geq u] \leq \bbP[ R_{i(N)} \geq s \ln u] + \bbP \Big[\sum_{i=1}^{i(N)}U_i \geq \frac{u}{sK(c)\ln u}\Big]
\end{equation}

By Lemma~\ref{R_N}, and using the fact that the sequence $R_m$ is increasing, we have that

\begin{equation}
\bbP[ R_{i(N)} \geq s \ln u]  \leq \bbP[ R_N \geq s \ln u] \leq \frac{CN}{u^{s \ln \lambda}}
\end{equation}

So that, choosing the parameter $s \geq \frac{\theta}{\ln \lambda}$, we see that

\begin{equation}\label{R}
\bbP[R_N \geq s \ln u] \leq \frac{CN}{u^{\theta}}.
\end{equation}

Moreover, using the fact that the random variables $U$ are non-negative, we have that

\begin{equation}
\bbP \Big[\sum_{i=1}^{i(N)}U_i \geq \frac{u}{sK(c)\ln u}\Big] \leq \bbP \Big[\sum_{i=1}^{N}U_i \geq \frac{u}{sK(c)\ln u}\Big]
\end{equation}

Now, if $\theta<1$, by Lemma~\ref{Borovkov}, we have that

\begin{equation}\label{U}
\bbP \Big[\sum_{i=1}^{i(N)}U_i \geq \frac{u}{sK(c)\ln u}\Big]  \leq \frac{CN( \ln u)^{\theta}}{u^{\theta}}
\end{equation}

Using inequalities~\eqref{T-R-U}, ~\eqref{R} and ~\eqref{U}, we see that

\begin{equation}
\bbP[T_N \geq u] \leq \frac{CN (\ln u)^{\theta}}{u^{\theta}}.
\end{equation}

This proves the first statement of Theorem~\ref{T_N2}.

If $\theta >1$, let $u > C_0 N \ln N$ for $C_0$ large enough, then by Lemma~\ref{Borovkov}, we have that

\begin{equation}\label{U2}
\bbP \Big[\sum_{i=1}^{i(N)}U_i \geq \frac{u}{sK(c)\ln u}\Big]   \leq \frac{CN (\ln u)^{\theta}}{u^{\theta}}
\end{equation}

Using inequalities~\eqref{T-R-U}, ~\eqref{R} and ~\eqref{U2}, we see that, for $u > C N \ln N$,

\begin{equation}
\bbP[T_N \geq u] \leq \frac{CN (\ln u)^{\theta}}{u^{\theta}}
\end{equation}

This proves the second statement of theorem~\ref{T_N2}.

\end{proof}%[End of proof of Theorem~\ref{T_N2}]

We now estimate the tail of the random variable $X_N$.

\begin{lemma}\label{X_N}
There exist $k>0$ and $\gamma_0>1$, so that for any $\gamma<\gamma_0$ and $\delta=\theta\gamma/(2\gamma+\theta)$, 
\begin{equation}
\bbP[ X_N \geq u] \leq \frac{C}{u^{\delta}}   e^{-\frac{\delta}{\gamma}kN}
\end{equation}

\end{lemma}

\begin{proof}[Proof of Lemma~\ref{X_N}]
We will rely on the following simple consequences of assumptions H2, H3 and H4 from the introduction. Recall
\begin{enumerate}
\item[H2.]
We assume that $\mu$ has finite moments, i.e. there exists a $t>1$ such that, 
$\int |z|^t \d\nu(z) = \bbE[|d_{\ell}|^t] < \infty$.
\item[H3.]
$m<0$.
\item[H4.] There exist $C>0$ and $\rho \in (0,1]$ such that, for all $h>0$,
\begin{equation}
\sup_{x >0} \nu(\{z \in \C | x <  |z| < x+h\}) \leq C h^{\rho}.
\end{equation} 
\end{enumerate}

Let 
\begin{equation}
B = \max \{ t \geq 0, \bbE[ |d_{\ell}|^t] < \infty \}
\end{equation}

We know, by Hypothesis H2, that $B>1$. 
If $B <\infty$, we define $\gamma_0 (\lambda)= \frac{\rho(B-2\lambda)}{B+\rho}$.
If $B = \infty$, we can choose $\gamma_0 (\lambda)= \rho$

Let $\nu$ be the distribution of $2\log|d_{\ell}|$.
Introduce, for $\lambda < \frac{B}{2}$, the exponential moment $M(\lambda)= \int e^{\lambda x} \nu(dx) = \bbE[ |d_{\ell}|^{2\lambda}]$ and the tilted measure
\begin{equation}
\nu_{\lambda}(dx) = e^{\lambda x - \log M(\lambda)} \nu(dx)
\end{equation}

\begin{lemma}\label{locreg}
For any  $\lambda < \frac{B}{2}$, any $\gamma<\gamma_0$, any $x>0$ and $h>0$, and any integer $N \geq 1$
\begin{equation}
\sup_{x >0} \nu_{\lambda}^{\ast N}(x,x+h) \leq C h^{\gamma}
\end{equation}
\end{lemma}

\begin{proof}[Proof of Lemma~\ref{locreg}]
We treat the case $B<\infty$. The case where $B=\infty$ is an immediate consequence.
We begin by proving the lemma for $\lambda=0$ and  $N=1$.
Obviously,
\begin{equation}
\nu_0(x,x+h) = \nu(x,x+h)= \bbP[ e^x< |d|^2 < e^x e^h]
\end{equation}
By Holder's inequality, for any pair of conjugate exponents $p$ and $q$
\begin{equation}
\nu(x,x+h)  \leq \bbP[ |d| > e^{\frac{x}{2}}]^{\frac{1}{p}}\bbP[ e^{\frac{x}{2}} < |d| < e^{\frac{x}{2}}e^{\frac{h}{2}}]^{\frac{1}{q}}
\end{equation}
By Markov inequality, for any $t<B$,
\begin{equation}
\bbP[ |d| > e^{\frac{x}{2}}] \leq C e^{-t \frac{x}{2}}
\end{equation}
Moreover by Hypothesis H4
\begin{equation}
\bbP[ e^{\frac{x}{2}} < |d| < e^{\frac{x}{2}}e^{\frac{h}{2}}] \leq C e^{\rho \frac{x}{2}} (e^{\frac{h}{2}} -1)^{\rho}
\end{equation}
So that, choosing $p= 1 + \frac{t}{\rho}$, we have that
\begin{equation}
\nu(x,x+h)  \leq C (e^{\frac{h}{2}} -1)^{\frac{t\rho}{t+\rho}}
\end{equation}
For $\gamma < \gamma_0$ one can choose $t < B$ so that $\frac{t\rho}{t+\rho} > \gamma$, which implies the Lemma when $N=1$ and $\lambda=0$.

We now prove the Lemma for $N=1$ and any $\lambda < B/2$.
For any pair of conjugate exponents $p$ and $q$, Holder's inequality yields that, 
\begin{equation}
\nu_{\lambda} (x, x+h) = \int e^{\lambda y} \1_{x<y<x+h} d\nu(y) \leq C  (\int e^{q \lambda y} \1_{x<y<x+h} d\nu(y))^{\frac{1}{q}} \nu(x,x+h)^{\frac{1}{p}}
\end{equation}
Choosing $q < \frac{B}{2\lambda}$ shows that
\begin{equation}
\nu_{\lambda} (x, x+h) \leq C(q) \nu(x,x+h)^{\frac{1}{p}}
\end{equation}
So that
\begin{equation}
\sup_{x>0} \nu_{\lambda} (x, x+h)  \leq C(q) \sup_{x>0} \nu(x,x+h)^{\frac{1}{p}} 
\end{equation}
For any $\gamma < \gamma_0$, by choosing $q$ close enough to $\frac{B}{2\lambda}$, we see that
\begin{equation}
\sup_{x>0} \nu_{\lambda} (x, x+h)  \leq C(q) h^{\gamma}
\end{equation}
Which proves the lemma for any $\lambda < B/2$ and $N=1$.

To deal with the case where $N >1$, we will now use the following simple remark, which clearly completes the proof of Lemma~\ref{locreg}
For any Borel set $E$ in $\R$, and any probability measure $\alpha$ on $\R$
\begin{equation}
\sup_{x \in \R} \alpha^{\ast N} (E-x) \leq \sup_{x \in \R} \alpha(E-x)
\end{equation}

Indeed, if $\mu_1$ and $\mu_2$ are any two probability measures on $\R$, then, 
\begin{equation}
\mu_1 \ast \mu_2(E) = \int \mu_1(E-x) d\mu_2(x) \leq \sup_{x \in \R} \mu_1(E-x)
\end{equation}
So that
\begin{equation}
\sup_{x \in \R} \mu_1 \ast \mu_2(E-x) \leq \min (\sup_{x \in \R} \mu_1(E-x), \sup_{x \in \R} \mu_2(E-x))
\end{equation}
By induction, for any probability measure $\alpha$ on $\R$, and any integer $N$
\begin{equation}
\sup_{x \in \R} \alpha^{\ast N} (E-x) \leq \sup_{x \in \R} \alpha(E-x)
\end{equation}
\end{proof}

Choose now $\lambda_0$ as the unique solution of 
\begin{equation}
\int x e^{\lambda_0 x} \nu(dx) =0 
\end{equation}

Then it is easy to see that $\lambda_0 > 0$ and $M(\lambda_0) <1$. 
Define $k = - \log M(\lambda_0) > 0$.

We now prove the following lemma
\begin{lemma}\label{tailSN}
For any $x>0$ and $h>0$, and any $\gamma < \gamma_0(\lambda_0)$
\begin{equation}
\bbP[ S_N \in (x,x+h)]  \leq C e^{-kN - \lambda_0 x} h^{\gamma}
\end{equation}
\end{lemma}

\begin{proof}
If $S_N^{\lambda}$ is the sum of N i.i.d random variables with distribution $\nu_{\lambda}$, we have seen that, for any real numbers $x>0$ and $h>0$,
\begin{equation}
\bbP[ S_N^{\lambda_0} \in (x,x+h)] \leq C h^{\gamma}
\end{equation}

And finally, since
\begin{equation}
\bbP[ S_N \in (x,x+h)] = e^{-kN} \bbE [e^{-\lambda_0 S_N^{\lambda_0}} \1_{S_N^{\lambda_0} \in (x,x+h)}] \leq e^{-kN - \lambda_0 x} \bbP[ S_N^{\lambda} \in (x,x+h)]
\end{equation}

We get

\begin{equation}
\bbP[ S_N \in (x,x+h)] \leq C e^{-kN - \lambda_0 x} h^{\gamma}
\end{equation}

Which proves Lemma~\ref{tailSN}
\end{proof}

We are now ready to come back to the estimation of the lower tail of the random variable $c_0(D)$.
\begin{lemma}\label{c0}
There exists a $k>0$, such that, for any $\epsilon >0$ and $v < 1-\epsilon$ 
\begin{equation}
\bbP [ c_0(D) \leq v] \leq C e^{-kN} v^{\frac{\gamma}{2}}
\end{equation}
\end{lemma}

\begin{proof}[Proof of Lemma~\ref{c0}]

Recall that
\begin{equation}
c_0(D) = | (-1)^N - \prod_{\ell=1}^N d_{\ell}|^2
\end{equation}
Thus,
\begin{equation}
\bbP [ c_0(D) \leq v^2] \leq \bbP[ 1-v \leq |\prod_{\ell=1}^N d_{\ell}| \leq 1+v] 
\end{equation}

Recalling that we have denoted by $S_N = 2 \sum_{\ell=1}^N \log |d_{\ell}|$
we see that
\begin{equation}
\bbP [ c_0(D) \leq v^2] \leq  \bbP[1-v \leq e^{\frac{S_N}{2}} \leq 1+v] \leq \bbP[ 2 \log (1-v) \leq S_N \leq 2\log(1+v)]
\end{equation}

By Lemma~\ref{tailSN}, we see that

\begin{equation}
\bbP [ c_0(D) \leq v^2] \leq C e^{-kN} v^{\gamma}
\end{equation}

which proves Lemma~\ref{c0}.
\end{proof}%[End of Proof of Lemma~\ref{c0}]

We are now able to prove Lemma~\ref{X_N}.
For any $0< \alpha <1$, and any $0< k'<k$, a simple union bound shows that
\begin{equation}
\bbP[ X_N \geq u] \leq \bbP\big[(1+U_\infty)(1+\hat U_\infty)\geq u^{1-\alpha} e^{\frac{k'N}{\gamma}}\big] +\bbP\big[\frac{1}{c_0(D)} \geq u^{\alpha}e^{-\frac{k'N}{\gamma}}\big]
\end{equation}

Since $\hat U_\infty$ has the same distribution than $U_\infty$, another simple union bound yields
\begin{equation}
\bbP\big[(1+U_\infty)(1+\hat U_\infty)\geq u^{1-\alpha} e^{\frac{k'N}{\gamma}}\big] \leq 2 \bbP\Big[(1+U_\infty)\geq u^{\frac{1-\alpha}{2}} e^{\frac{k'N}{2\gamma}}\Big]
\end{equation}

We know by Lemma~\ref{tailUinfty} that 
\begin{equation}
\bbP [U_{\infty} \geq u] \sim \frac{C}{u^{\theta}}
\end{equation}

So that, by Lemma ~\ref{c0}
\begin{equation}
\bbP[ X_N \geq u] \leq Cu^{-\frac{(1-\alpha)\theta}{2}}e^{-\frac{k'\theta}{2\gamma}N} + C e^{-(k-k')N} u^{-\alpha\gamma}      
\end{equation}

Now, choosing $\alpha = \frac{\theta}{2\gamma + \theta}$, and $k'= \frac{2\gamma}{\theta+2\gamma}k$, we see that

\begin{equation}
\bbP[ X_N \geq u] \leq \frac{C}{u^{\delta}}   e^{-\frac{\delta}{\gamma}kN}
\end{equation}

where 
\begin{equation}
\delta = \frac{\theta\gamma}{2\gamma+\theta}
\end{equation}

\end{proof}%[End of Proof of Lemma~\ref{X_N}]

\end{document}